\documentclass[12pt,reqno]{amsart}

\usepackage{amsmath,amssymb,amscd,graphicx}
\usepackage[all,2cell]{xy}
\usepackage{ifthen}
\usepackage{subsupscripts}
\usepackage[margin=1.3in]{geometry}	
\usepackage{url}
\usepackage{cite}
\usepackage{tikz}
\usepackage[linktocpage=true]{hyperref}

\newboolean{sections}
\setboolean{sections}{true}

\UseComputerModernTips
\UseTwocells

\def\eoe{\unskip\ \hglue0mm\hfill$\diamond$\smallskip\goodbreak}

\newcommand{\calA}{\mathcal{A}}

\newcommand{\calG}{\mathcal{G}}
\newcommand{\calH}{\mathcal{H}}
\newcommand{\calI}{\mathcal{I}}

\newcommand{\LL}{{\mathbb{L}}}  
\newcommand{\RR}{{\mathbb{R}}}  
\newcommand{\UU}{{\mathbb{U}}}  
\newcommand{\VV}{{\mathbb{V}}}  
\newcommand{\WW}{{\mathbb{W}}}  
\newcommand{\XX}{{\mathbb{X}}}  
\newcommand{\ZZ}{{\mathbb{Z}}}  

\newcommand{\curv}{{\operatorname{curv}}}   
\newcommand{\id}{{\operatorname{id}}}  
\newcommand{\pr}{{\operatorname{pr}}} 

\newcommand{\g}{{\mathfrak{g}}} 
\newcommand{\lmc}{\theta^L}		
\newcommand{\rmc}{\theta^R}		

\newcommand{\LieButterfly}{\mathsf{2TermL}_\infty^\flat} 
\newcommand{\LieBC}{\mathsf{2TermL}_\infty} 

\newcommand{\toto}{{~\rightrightarrows~}} 

\newcommand{\genz}{\frac{\partial}{\partial \theta}}  

\newcommand{\dashto}{{\;\dashrightarrow\;}} 

\makeatletter						
\newtheorem*{rep@theorem}{\rep@title}
\newcommand{\newreptheorem}[2]{
\newenvironment{rep#1}[1]{
\def\rep@title{#2 ##1}
\begin{rep@theorem}}
{\end{rep@theorem}}}
\makeatother

\newcommand{\ifsection}[2]{\ifthenelse{\boolean{sections}}{#1}{#2}}

\theoremstyle{plain}
\ifsection{
    \newtheorem{theorem}{Theorem}[section]
    
	\numberwithin{equation}{section}
	\numberwithin{figure}{section}
	\usepackage{chngcntr}
	\counterwithin{table}{section}
}
{
    \newtheorem{theorem}{Theorem}
}

\newreptheorem{theorem}{Theorem}		
\newreptheorem{proposition}{Proposition}
\newreptheorem{corollary}{Corollary}
\newtheorem{proposition}[theorem]{Proposition}

\newtheorem{corollary}[theorem]{Corollary}
\newtheorem{lemma}[theorem]{Lemma}

\theoremstyle{definition}
\newtheorem{definition}[theorem]{Definition}
\newtheorem{example}[theorem]{Example}
\newtheorem{remark}[theorem]{Remark}

\newboolean{workmode}
\setboolean{workmode}{true}

\newcommand{\ifwork}[1]{\ifthenelse{\boolean{workmode}}{#1}{}}
\newcommand{\comment}[1]{}
\newcommand{\mute}[1]{}
\newcommand{\printname}[1]{}

\ifwork{
\renewcommand{\comment}[1]{{\marginpar{*}\ \scriptsize{#1}\ }}
\renewcommand{\mute}[1]{{\scriptsize \ #1\ }\marginpar{\scriptsize muted}}
\renewcommand{\printname}[1]
    {\smash{\makebox[0pt]{\hspace{-1.0in}\raisebox{8pt}{\tiny #1}}}}
}

\author{Derek Krepski}
\address{
Department of Mathematics, University of Manitoba, 
Winnipeg, MB, Canada 
}
\email{\href{mailto:Derek.Krepski@umanitoba.ca}{Derek.Krepski@umanitoba.ca}}

\author{Jennifer Vaughan}
\address{
Department of Mathematics, University of Manitoba, 
Winnipeg, MB, Canada 
}
\email{\href{mailto:Jennifer.Vaughan@umanitoba.ca}{Jennifer.Vaughan@umanitoba.ca}}

\title[Multiplicative vector fields on bundle gerbes]{Multiplicative vector fields on bundle gerbes}
\date{\today}

\thanks{The authors acknowledge the support of the Natural Sciences and Engineering Research Council of Canada (RGPIN-2015-05833), and the Pacific Institute of Mathematical Sciences. }

\begin{document}

\begin{abstract}

Infinitesimal symmetries of $S^1$-bundle gerbes are modelled by multiplicative vector fields on Lie groupoids.  It is shown that a connective structure on a bundle gerbe gives rise to a natural horizontal lift of multiplicative vector fields to the bundle gerbe, and that the 3-curvature presents the obstruction to the horizontal lift being a morphism of Lie 2-algebras.  Connection-preserving multiplicative vector fields on a bundle gerbe with connective structure are shown to inherit a natural Lie 2-algebra structure; moreover, this Lie 2-algebra is canonically quasi-isomorphic to the Poisson-Lie 2-algebra of the 2-plectic base manifold $(M,\chi)$, where $\chi$ is the 3-curvature of the connective structure.  As an application of this result, we give  analogues of a formula of Kostant in the 2-plectic and quasi-Hamiltonian contexts.

\end{abstract}

\maketitle

\section{Introduction} \label{s:intro}

Let $M$ be a smooth manifold. A well-known Theorem of Weil \cite{weil1952} shows that the set of  isomorphism classes of principal $S^1$-bundles over $M$ are in one-to-one correspondence with $H^2(M;\ZZ)$. The analogous objects corresponding to cohomology classes in  $H^3(M;\ZZ)$ are known as $S^1$-gerbes,  introduced in differential geometry by Brylinski \cite{brylinski2007loop} using the formalism of stacks (presheaves of groupoids), following Giraud \cite{giraud1971}. Other models for $S^1$-gerbes include $S^1$-central extensions of Lie groupoids (e.g., see Behrend-Xu \cite{behrend2011differentiable}), $S^1$-bundle gerbes (equivalent to $S^1$-central extensions of submersion groupoids), due to Murray \cite{murray1996bundle} (see also Hitchin \cite{hitchin2001lectures} and Chatterjee \cite{chatterjee1998construction}), Dixmier-Douady bundles \cite{dixmier1963champs}, differential cocycles \cite{krepski2018differential}, and principal Lie 2-group bundles \cite{baez2007higher,nikolaus2013four,wockel2011principal}. 

The present paper primarily adopts the $S^1$-bundle gerbe perspective (see Section \ref{ss:gerbes} for a review of bundle gerbes) to establish some results for $S^1$-gerbes that are analogous to well-known results for their degree-2 counterparts, principal $S^1$-bundles. Moreover, as described in more detail below, these results, and fundamentally the backdrop of multiplicative vector fields, lead naturally into further results on `higher structures' in differential geometry---namely, 2-plectic geometry \cite{rogers20132plectic}, homotopy moment maps \cite{callies2016homotopy}, and group-valued moment maps \cite{alekseev1998lie}. 
Indeed, we explain how multiplicative vector fields encode (connection preserving) infinitesimal symmetries of $S^1$-bundle gerbes (with connective structure); that the underlying Lie 2-algebra of such symmetries is quasi-isomorphic to the Poisson Lie 2-algebra of the base manifold; and that multiplicative vector fields provide a natural setting  for 2-plectic and quasi-Hamiltonian analogues of a formula of Kostant from symplectic geometry.

\bigskip

To motivate our discussion, consider a principal $S^1$-bundle $P\to M$ with connection $\gamma\in \Omega^1(P)$ and let $\mathsf{curv}(\gamma) \in \Omega^2(M)$ denote the curvature of the connection. It is well known that the connection determines a horizontal lift of vector fields on $M$ to vector fields on $P$,
\[
\mathsf{Lift}_\gamma: \mathfrak{X}(M) \to \mathfrak{X}(P),
\]
 which is a Lie algebra homomorphism if and only if $\mathsf{curv}(\gamma)=0$. 
 What is the  analogue for $S^1$-bundle gerbes? 
 
 To  arrive at the analogue, we employ \emph{multiplicative vector fields} on Lie groupoids. Recall  from \cite{mackenzie1998classical} that a multiplicative vector field on a Lie groupoid $\mathbf{G}=\{ G_1 \toto G_0 \}$ is a Lie groupoid morphism (i.e., a functor) $\mathbf{x}: \mathbf{G} \to T\mathbf{G} = \{ TG_1 \toto TG_0 \}$ such that $\pi_{\mathbf{G}} \circ \mathbf{x} = \id_{\mathbf{G}}$ where $\pi_{\mathbf{G}}: T\mathbf{G} \to \mathbf{G}$ is the tangent bundle projection functor. Hence a multiplicative vector field $\mathbf{x}$ is described by a pair $(\mathbf{x}_0, \mathbf{x}_1)$ of vector fields, $\mathbf{x}_0 \in \mathfrak{X}(G_0)$ and $\mathbf{x}_1\in \mathfrak{X}(G_1)$,  that are compatible with the Lie groupoid structure. The   multiplicative vector fields  on $\mathbf{G}$ are the objects of a category $\mathbb{X}(\mathbf{G})$, whose morphisms are natural transformations $\eta:\mathbf{x} \Rightarrow \mathbf{x}'$ satisfying $\pi_\mathbf{G}(\eta(x))=\epsilon(x)$ for every object $x\in G_0$, where $\epsilon:G_0\to G_1$ denotes the unit map.  
 By a Theorem of Hepworth  \cite[Theorem 4.15]{hepworth2009vector}, multiplicative vector fields on a Lie groupoid $\mathbf{G}$ correspond to vector fields on the associated differentiable stack $B\mathbf{G}$. 
 
An $S^1$-gerbe on a smooth manifold $M$ is a differentiable stack, which may be presented as a bundle gerbe over $M$, or equivalently an $S^1$-central extension 
$ P \toto X$ of the submersion Lie groupoid $X\times_M X \toto X$ associated to a surjective submersion $X\to M$ (see Section \ref{ss:centext}). Therefore, vector fields on the $S^1$-gerbe may be modelled by multiplicative vector fields on the Lie groupoid $P\toto X$. In fact, as we show in Theorem \ref{t:equiv-mvf-col}, the category $\mathbb{X}(P\toto X)$ of multiplicative vector fields on $P\toto X$ is equivalent to the category described by Collier \cite{collier-PhDThesis2012} of \emph{infinitesimal symmetries} of the corresponding $S^1$-gerbe, viewed as a stack. Accordingly, a principal objective of this work is to illustrate how multiplicative vector fields naturally model infinitesimal symmetries of bundle gerbes, independent of the machinery of stacks.

Additionally, by a recent result of Berwick-Evans and Lerman \cite{berwick2016lie}, multiplicative vector fields on a Lie groupoid naturally have a structure of a Lie 2-algebra. In this work, a Lie 2-algebra is a 2-term $L_\infty$-algebra (e.g., see Baez-Crans \cite{baez2004higher})---that is, a 2-term chain complex $L_1 \to L_0$, equipped with a graded bracket and a chain homotopy coherently measuring the failure of the bracket to satisfy the Jacobi identity. We work within a localization of the bicategory of 2-term $L_\infty$-algebras with respect to quasi-isomorphisms---namely, Noohi's bicategory of Lie 2-algebras with \emph{butterflies} as 1-morphisms  (see Section \ref{ss:Lie2defs} for a review of the bicategory of Lie 2-algebras).

We thus arrive at a natural setting in which to formulate the bundle gerbe analogue of the aforementioned result on horizontal lifts. Indeed, suppose an $S^1$-gerbe over $M$ is presented by a bundle gerbe $P\toto X$, an $S^1$-central extension of $X\times_M X \toto X$, where $X\to M$ is a surjective submersion. The Lie 2-algebra of multiplicative vector fields $\mathbb{X}(X\times_M X \toto X)$ is naturally quasi-isomorphic to the Lie algebra $\mathfrak{X}(M)$ of vector fields on $M$, and we show in Theorem \ref{p:3curv_obs} that a connective structure on  $P \toto X$ (i.e., a  connection $\gamma \in \Omega^1(P)$ and curving $B\in \Omega^2(X)$) gives rise to a functor
\[
\mathsf{Lift}_\gamma: \mathbb{X}(X\times_M X \toto X) \longrightarrow \mathbb{X}(P\toto X),
\]
which is a morphism of Lie 2-algebras if and only if the associated 3-curvature $3\text{-}\mathsf{curv}(B,\gamma) \in \Omega^3(M)$ of the connective structure vanishes. (\emph{Cf}.\ Remark \ref{r:3curvCollier}.)

Given a bundle gerbe $P\toto X$ with connective structure $(B,\gamma)$, we also consider  multiplicative vector fields $(\mathbf{x},\mathbf{p})$ that preserve the connective structure in an appropriate sense. Bearing in mind the definition of isomorphism for bundle gerbes with connective structure (see Definition \ref{d:gerbe1arrow}), we should not expect that the vector fields $\mathbf{x}$ and $\mathbf{p}$ preserve $B$ and $\gamma$ on the nose. 
Instead, we propose in Definition \ref{d:connexpresmultivf} that the multiplicative vector field  $(\mathbf{x},\mathbf{p})$   preserves the connective structure $(B,\gamma)$ if
\[
L_{\mathbf{x}} B = d\alpha, \quad \text{and} \quad L_{\mathbf{p}} \gamma= \delta \alpha, \quad \text{for some }\, \alpha\in \Omega^1(X),
\]
where $\delta = \Sigma (-1)^j \delta_j^*$ is the simplicial differential for $P\toto X$.
 We show that the collection $\mathbb{X}(P; B,\gamma)$ of connection-preserving multiplicative vector fields is a subcategory of $\mathbb{X}(P\toto X)$  (see Corollary \ref{c:subcat} for the precise statement), and it naturally inherits a Lie 2-algebra structure  (see Proposition \ref{p:multivfLie2}). Moreover, we verify in Theorem \ref{t:equiv-mvf-col-withconnex} that $\mathbb{X}(P; B,\gamma)$ is equivalent to the 2-category of connection-preserving infinitesimal symmetries of the gerbe (i.e., viewed as a stack) considered in \cite{collier2011infinitesimal,collier-PhDThesis2012}.

The Lie 2-algebra 
$\mathbb{X}(P; B,\gamma)$
plays a natural role in establishing some other results that may be viewed as 2-plectic and quasi-Hamiltonian analogues of classical results in symplectic geometry. 
Recall that 
for a 2-plectic manifold $(M,\chi)$, the analogue of the Poisson algebra of smooth functions on a symplectic manifold is the Poisson-Lie 2-algebra $\mathbb{L}(M,\chi)$ (see Definition \ref{d:rogers}, due to Rogers \cite{rogers20132plectic}). 
In Theorem \ref{t:invertbutterflymulti-rogers}, we show that $\mathbb{L}(M,\chi)$ is canonically quasi-isomorphic to the Lie 2-algebra $\mathbb{X}(P; B,\gamma)$ of connection-preserving multiplicative vector fields on $P\toto X$ whenever $\chi=3\text{-}\mathsf{curv}(B,\gamma)$. In other words, the Poisson Lie 2-algebra $\mathbb{L}(M,\chi)$ of a 2-plectic manifold is quasi-isomorphic to the Lie 2-algebra of infinitesimal symmetries of a \emph{prequantization} of $(M,\chi)$, a bundle gerbe $P\toto X$ over $M$, equipped with a connective structure $(B,\gamma)$ whose 3-curvature is $\chi$. Such infinitesimal symmetries are sometimes called \emph{infinitesimal quantomorphisms} in the literature---e.g., see  \cite{fiorenza2014algebras} wherein the authors prove a similar result (\emph{cf}.\  Remark \ref{r:FRS}); 
see also \cite{syvestre2020prequantisation} (\emph{cf}.\ Remark \ref{r:SW}).

As  applications of  the quasi-isomorphism $\mathbb{L}(M,\chi)\cong \mathbb{X}(P; B,\gamma)$ in Theorem \ref{t:invertbutterflymulti-rogers}, we propose 2-plectic and quasi-Hamiltonian analogues of a formula of Kostant.  Let $G$ be a compact connected Lie group with Lie algebra $\g$. Given a  Hamiltonian $G$-action on a symplectic manifold $(M,\omega)$, a formula due to Kostant gives a lift of the infinitesimal $\g$-action on $M$  to a $\g$-action on a prequantization of $(M,\omega)$, a principal $S^1$-bundle $P\to M$ with connection whose curvature equals $\omega$, by infinitesimal quantomorphisms. 
In Section \ref{s:kost}, we show in Corollary \ref{c:2KostantLift} that a Hamiltonian $G$-action on a 2-plectic manifold $(M,\chi)$ (i.e., one admitting  a \emph{homotopy moment map} as in \cite{callies2016homotopy}) naturally lifts to a `weak' infinitesimal $\g$-action on a prequantization of $(M,\chi)$ by infinitesimal quantomorphisms---i.e, a butterfly morphism 

\[
\g \dashto \mathbb{X}(P;B,\gamma),
\]
where $P\toto X$ is a bundle gerbe over $M$ with connective structure $(B,\gamma)$.
Moreover, if the bundle gerbe is a (strongly) $G$-equivariant bundle gerbe, we show in Theorem \ref{t:g-actsagree} that the above lift agrees with the underlying $\g$-actions on $P$ and $X$ by multiplicative vector fields that preserve $B$ and $\gamma$ on the nose (i.e. with $\alpha=0$). 
For quasi-Hamiltonian spaces with Lie group-valued moment maps, a similar lift is obtained in Theorem \ref{t:qHamKost}. 
These results help clarify some of the concluding questions raised in
\cite[Section 13]{callies2016homotopy}.

\bigskip

\subsection*{Organization of the paper} 

In \textbf{Section \ref{s:prelim}}, we review some preliminaries to establish notation used throughout the paper. 
In \textbf{Section \ref{s:gerbes}}, we recall some foundational aspects of bundle gerbes and begin our study of the categories of multiplicative vector fields  and  connection-preserving multiplicative vector fields on bundle gerbes. We provide a sketch of the equivalence between the category of multiplicative vector fields on a bundle gerbe and the category of infinitesimal symmetries of the underlying gerbe (Theorem \ref{t:equiv-mvf-col}), and similarly for multiplicative vector fields and infinitesimal symmetries that preserve a connective structure (Theorem \ref{t:equiv-mvf-col-withconnex}).
 In \textbf{Section \ref{s:lie2}}, we review Noohi's 2-category of Lie 2-algebras with butterfly morphisms. We establish Theorem \ref{p:3curv_obs} characterizing the 3-curvature as the obstruction to the horizontal lift defining a morphism of Lie 2-algebras. We also introduce (Proposition \ref{p:multivfLie2}) the natural Lie 2-algebra structure of connection-preserving multiplicative vector fields on a bundle gerbe.
In \textbf{Section \ref{s:app}}, we establish some results in 2-plectic geometry and quasi-Hamiltonian geometry that are analogous to some classical results in symplectic geometry. In particular, we show (Theorem \ref{t:invertbutterflymulti-rogers}) that the Lie 2-algebra of  multiplicative vector fields preserving the connective structure on a bundle gerbe is canonically quasi-isomorphic to the Poisson-Lie-2-algebra of the 2-plectic manifold $(M,\chi)$, where $\chi$ denotes the 3-curvature of the connective structure. We conclude with a 2-plectic analogue (Corollary \ref{c:2KostantLift}) and a quasi-Hamiltonian analogue (Theorem \ref{t:qHamKost}) of a formula in symplectic geometry due to Kostant that lifts Hamiltonian Lie-algebra actions to infinitesimal quantomorphisms.

\section{Preliminaries and notation} \label{s:prelim}

Throughout this paper, all manifolds and maps between manifolds are assumed to be smooth.  We assume the reader has some familiarity with Lie groupoids, though we recall some aspects next to establish notation.
Recall that a Lie groupoid  $G_1 \toto G_0$ gives rise to a simplicial manifold $G_\bullet$:
for $k \geq 2$, write 
\[
G_k = \underbrace{G_1 \times_{G_0} G_1 \times_{G_0} \cdots \times_{G_0} G_1}_{k \text{ factors}}
\]
whose elements are $k$-tuples $(g_1, \ldots, g_k)$ of composable arrows (with $s(g_i)=t(g_{i+1})$).  For $0 \leq i \leq k$, let $\partial_i:G_k \to G_{k-1}$ be the \emph{face maps} given by
\[
\partial_i(g_1, \ldots, g_k) = \left\{ 
\begin{array}{ll}
(g_2, \ldots, g_k) & \text{if }i=0 \\
(g_1, \ldots, g_ig_{i+1},\ldots, g_k) & \text{if }0<i<k \\
(g_1, \ldots, g_{k-1}) & \text{if }i=k. \\
\end{array}
\right.
\]
For convenience, we set $\partial_0=s$ and $\partial_1=t$ on $G_1$. The face maps satisfy the simplicial identities
\begin{align} \label{eq:simplicialidentities}
\partial_i \partial_j & = \partial_{j-1} \partial_i \quad \text{for }i<j, 
\end{align}
and hence $\partial = \sum (-1)^i \partial_i^*$ defines a differential for the complexes $\Omega^k(G_{\bullet})$, $k\geq 0$. 

In particular, for a surjective submersion, $\pi: X\to M$, we may consider the \emph{submersion groupoid} $X\times_M X \toto X$, and obtain the simplicial manifold $X^{[\bullet+1]}$, where
\[
X^{[n]} = \{ (x_1, \ldots, x_{n}) \, | \, \pi(x_1) = \cdots =\pi(x_{n})\}.
\]
As shown in \cite{murray1996bundle}, the following complexes are exact:
\begin{align} \label{eq:deltaexact}
0 \longrightarrow \Omega^k(M) \stackrel{\pi^*}{\longrightarrow} \Omega^k(X) \stackrel{\partial}{\longrightarrow} \Omega^k(X^{[2]}) \stackrel{\partial}{\longrightarrow} \cdots
\end{align}
Similarly, for proper Lie groupoids $G_1 \toto G_0$, the sequence
\[
\Omega^0(G_0) \stackrel{\partial}{\longrightarrow} \Omega^0(G_1) \stackrel{\partial}{\longrightarrow} \Omega^0(G_2) \stackrel{\partial}{\longrightarrow} \cdots
\]
is exact \cite{crainic2003differentiable}.

For a compact Lie group $K$ acting on manifolds $X$ and $M$, and a $K$-equivariant surjective submersion $\pi:X\to M$, there are natural $K$-equivariant analogues of the above. Recall briefly that for a $K$-manifold $M$, the Cartan model of $K$-equivariant differential forms is given by $\Omega_K(M) = (S(\g^*)\otimes \Omega(M))^K$, the algebra of $K$-equivariant polynomials $\alpha:\g \to \Omega(M)$, with differential $(d_K\alpha)(\xi) = d\alpha(\xi) - \iota_{\xi_M}\alpha(\xi)$. Here, $\xi_M \in \mathfrak{X}(M)$ denotes the generating vector field for $\xi \in \g$: $\xi_M(m) = \frac{d}{dt}|_{t=0} \exp(-t\xi)\cdot m$. The map $\partial$ extends to $K$-equivariant forms, and the $K$-equivariant version of \eqref{eq:deltaexact} is also exact---see \cite{stienon2010equivariant}. 

Finally, for a principal $S^1$-bundle $P\to M$, we identify $\mathrm{Lie}(S^1) \cong \RR$ and write $\genz $ for the generating vector field of the principal $S^1$-action.

\section{$S^1$-gerbes over manifolds} \label{s:gerbes}

This section employs multiplicative vector fields on Lie groupoids as infinitesimal symmetries of $S^1$-bundle gerbes, and introduces their connection-preserving counterparts.  After a brief review of $S^1$-bundle gerbes, we consider some basic properties of  the category of  multiplicative vector fields on bundle gerbes, showing that the connection-preserving multiplicative vector fields form a full subcategory (Corollary \ref{c:subcat}). In Section \ref{ss:infsym}, we sketch an equivalence of categories between (connection-preserving) multiplicative vector fields on bundle gerbes (with connective structure) and the category of (connection-preserving) infinitesimal symmetries of $S^1$-banded gerbes \cite{collier2011infinitesimal}.

\subsection{Bundle gerbes with connective structure} \label{ss:gerbes}

We review some elementary aspects of  $S^1$-bundle gerbes over a manifold, following \cite{nikolaus2011equivariance}.  For a detailed introduction to bundle gerbes, we refer the reader to Murray's original paper \cite{murray1996bundle} as well as \cite{murray2000bundle} and \cite{stevenson2000geometry}. 

\begin{definition} \label{d:bgerbe}
Let $M$ be a manifold. An \emph{$S^1$-bundle gerbe} (or simply, \emph{bundle gerbe})  $\calG=(X, P, \mu)$ over $M$   consists of a surjective submersion $\pi:X\to M$, a principal $S^1$-bundle $P\to X^{[2]}$, and a bundle isomorphism $\mu:\partial_2^*P \otimes \partial_0^*P \to \partial_1^*P$  over $X^{[3]}$, called the gerbe product, satisfying the associativity condition $\partial_1^*\mu \circ (\partial_3^*\mu\otimes \id) = \partial_2^*\mu \circ (\id \otimes \partial_0^*\mu)$ over $X^{[4]}$.
\end{definition}

\begin{definition} \label{d:connex}
Let $\calG=(X, P, \mu)$ be a bundle gerbe over a manifold $M$. A \emph{connection}
on $\calG$ is a connection 1-form $\gamma \in \Omega^1(P)$ such that the bundle isomorphism $\mu$ is connection-preserving: $\mu^*(\hat{\partial}_1^* \gamma) = \hat{\partial}_2^* \gamma \otimes \hat{\partial}_0^*\gamma$, where $\hat{\partial}_j:\partial_j^*P \to P$ denotes the natural map covering $\partial_j$. A \emph{curving} for the connection $\gamma$ is a 2-form $B\in \Omega^2(X)$ satisfying  $\partial B=\curv(\gamma)$.  
We will call the pair $(B,\gamma)$ a \emph{connective structure} for $\calG$.
\end{definition}

Since $\partial dB = d \partial B = d\curv(\gamma)=0$, there exists a unique closed form $\chi \in \Omega^3(M)$, called the \emph{3-curvature} of the connective structure, satisfying $\pi^*\chi = dB$.

\begin{example} \label{eg:trivialgerbe}
The trivial bundle gerbe $\calI$ over $M$ is the bundle gerbe with $\pi=\id_M$, $P=M\times S^1$, and canonical gerbe product $\mu$.
For $\omega \in \Omega^2(M)$,  $\calI_{\omega}$, is the bundle gerbe $\calI$ with trivial connection on $M\times S^1 \to M$ and curving $\omega$. Note that the 3-curvature of this trivial connection with curving $\omega$ is $d\omega$. 
\eoe
\end{example}

\begin{example} \label{eg:cechdata}
Let $\{U_i\}$ denote a good open cover of a smooth manifold $M$, and let $X=\bigsqcup U_i \to M$ denote the map whose restrictions $U_i \hookrightarrow M$ are the inclusions. Then we may identify $X^{[2]} = \bigsqcup U_{ij}$, where $U_{ij}$ denotes the intersection $U_i\cap U_j$ (and similarly for $X^{[n]}$ and $n$-tuples of indices, e.g., $X^{[3]} = \bigsqcup U_{ijk}$).  Suppose $\{g_{ijk}:U_{ijk}\to S^1\}$ denotes a collection of smooth functions. Let  $P=\bigsqcup U_{ij} \times S^1 \to \bigsqcup U_{ij}$ and set 
\[
\mu(x,z,w) = (x,zwg_{ijk}(x)),
\]
where we are identifying $x\in M$ with its copies in the  $U_{ij}$'s. It's straightforward to see that $\mu$ defines a bundle gerbe product as in Definition \ref{d:bgerbe}, provided $g_{ijk}$ defines a \v{C}ech 2-cocycle.  Denote the resulting bundle gerbe by $\calG_{g_{ijk}}$.

A connective structure on such a gerbe is thus specified by a collection  $\{\gamma_{ij} = \pr^*A_{ij} + z^{-1}dz \}$, or more simply by the forms $A_{ij} \in \Omega^1(U_{ij})$ that must satisfy $A_{jk} - A_{ik} + A_{ij} = g_{ijk}^{-1} dg_{ijk}$ on each $U_{ijk}$, together with a collection $B_i\in \Omega^2(U_i)$ satisfying $B_j-B_i = dA_{ij}$ on each $U_{ij}$. In other words, a bundle gerbe with connective structure with respect to the covering $\{U_i\}$ is specified by a  \v{C}ech-Deligne 2-cocycle $(g,A,B)$ \cite{brylinski2007loop}.
\eoe
\end{example}

Recall that a \emph{common refinement} of two surjective submersions $\pi: X \to M$ and $\pi': X' \to M$ is a surjective submersion $\varpi:Y\to M$ together with surjective submersions $\lambda: Y \to X$ and $\rho: Y\to X'$ such that $\pi \circ \lambda = \varpi = \pi' \circ \rho$.  For example, the fibre product $Y=X\times_M X'$ is such a common refinement. 

\begin{definition}\label{d:gerbe1arrow} \cite{nikolaus2011equivariance} Let $\calG=(X,P,\mu)$ and $\calG'=(X',P',\mu')$ be bundle gerbes over a manifold $M$.
An \emph{isomorphism  $\calG \dashto \calG'$ of bundle gerbes}  consists of a common refinement $Y\to M$ of $\pi$ and $\pi'$, together with a principal  $S^1$-bundle $Q \to Y$ and an isomorphism of $S^1$-bundles $\varphi:\rho^*P' \otimes \partial_0^*Q \to \partial_1^*Q \otimes \lambda^* P$ over $Y^{[2]}$ satisfying the coherence condition
\[ 
(\id \otimes \lambda^*\mu) \circ (\partial_2^* \varphi \otimes \id) \circ (\id \otimes \partial_0^*\varphi) = \partial_1^*\varphi \circ (\rho^*\mu'\otimes \id)
\]
over $Y^{[3]}$. We will denote such an isomorphism by $(Q,\varphi)$, or sometimes simply $Q$.

An \emph{isomorphism of bundle gerbes with connective structure} $(\calG,B,\gamma) \dashto (\calG', B', \gamma')$ is an isomorphism $(Q,\varphi)$ as above, together with a connection 1-form $\gamma_Q \in \Omega^1(Q)$ such that $\curv(\gamma_Q) = \rho^*B' - \lambda^*B$ and $\varphi$ is connection-preserving. We will denote such an isomorphism by $(Q,\varphi,\gamma_Q)$, or sometimes simply $(Q,\gamma_Q)$.
\end{definition}

Given two isomorphisms $(Q,\gamma_Q), (Q',\gamma_{Q'}): (\calG,B,\gamma) \dashto (\calG', B', \gamma')$ of bundle gerbes with connective structure, there is a notion of 2-isomorphism $(Q,\gamma_Q) \Rightarrow (Q',\gamma_{Q'})$; consequently bundle gerbes over $M$ with connective structure form a bicategory $\mathsf{BunGerbe}^\nabla(M)$ . (See \cite{nikolaus2011equivariance} for details.)

Bundle gerbes  on a manifold $M$ are classified by their \emph{Dixmier-Douady} class (or DD-class) in $ H^3(M;\ZZ)$, analogous to the Chern-Weil classification of principal $S^1$-bundles by their Chern class in $H^2(M;\ZZ)$.  Moreover, the DD-class maps to the cohomology class of the 3-curvature via the coefficient homomorphism $H^3(M;\ZZ) \to H^3(M;\RR) \cong H^3_{dR}(M)$. For  bundle gerbes on $M$ with connective structure, the classification takes values in  the group of  differential characters of degree 3 (e.g., see \cite{krepski2018differential}).

Let $G$ be a compact Lie group. 
For manifolds equipped with a smooth $G$-action, there are various notions of $G$-equivariant bundle gerbes in the literature---see, for instance, \cite{meinrenken2003basic}, \cite{stienon2010equivariant}, \cite{murray2016equivariant}, \cite{nikolaus2011equivariance}. Moreover, there are $G$-equivariant analogues of the aforementioned classifications: $G$-equivariant bundle gerbes are classified by their $G$-equivariant DD-class, while $G$-equivariant bundle gerbes with connective structure are classified by $G$-equivariant differential characters.

\begin{example} \label{e:basicgerbe}
An important family of examples of bundle gerbes naturally occur on Lie groups. 
Let $G$ be a compact simple Lie group, with an invariant inner product $\langle-,-\rangle$ on its Lie algebra $\g$.  Let $\lmc$ and $\rmc$ denote the left and right invariant Maurer-Cartan forms. Let $\eta = \frac{1}{12} \langle \lmc, [\lmc, \lmc] \rangle$ denote the Cartan 3-form, and recall  that $\eta$ has an equivariant extension 
\(
\eta_G(\xi) = \eta +  \frac{1}{2} \langle \lmc + \rmc, \xi \rangle \in \Omega^3_G(G),
\)
with respect to the conjugation action.
For certain choices of inner product (parameterized by the so-called \emph{level}), there are a number of explicit constructions in the literature of a bundle gerbe with connective structure $(\calG, B, \gamma)$  over $G$ whose 3-curvature is $\eta$, as well as $G$-equivariant versions whose $G$-equivariant 3-curvature is $\eta_G$---e.g., see \cite{meinrenken2003basic}, \cite{gawedzki2004basic} and \cite{krepski2018basic}. 
\eoe
\end{example}

\subsection{Bundle gerbes as $S^1$-central extensions} \label{ss:centext}

Bundle gerbes may be recast as $S^1$-central extensions of Lie groupoids. This interpretation is briefly reviewed next. Given a bundle gerbe $\calG = (X,P,\mu)$, we consider the Lie groupoid $P \toto X$, with source and target maps, $s=\pi_0:= \partial_0 \circ \pi_P$ and $t=:\pi_1 = \partial_1 \circ \pi_P$, respectively, where $\pi_P:P\to X^{[2]}$ denotes the bundle projection. The groupoid multiplication $m:P\times_X P \to P$ (also denoted by $\cdot$) is induced by the gerbe product,
\begin{align*}
\mu(x,y,z;p\otimes q) = (x,y,z;m(p,q)),
\end{align*}
as are the unit $\epsilon:X \to P$ and inverse $i:P\to P$, by setting
\begin{align*}
\mu^{-1}(x,x,y;p) = (x,x,y; \epsilon(x)\otimes p), \quad \text{and} \quad  \mu^{-1} (x,y,x; \epsilon(x)) = (x,y,x; p\otimes i(p))
\end{align*}
for all $p\in P$ lying over $(x,y)\in X^{[2]}$.
In this viewpoint, we see that $P\toto X$ is an \emph{$S^1$-central  extension} of the submersion groupoid $X^{[2]} \toto X$. That is, the bundle projection $P\to X^{[2]}$ is a morphism of Lie groupoids,
\[
\xymatrix{
P  \ar@<-.7ex>[d]\ar@<.7ex>[d]  \ar[r] & X^{[2]}  \ar@<-.7ex>[d]\ar@<.7ex>[d] \\
X \ar@{=}[r] & X
}
\]
and the $S^1$-action on $P$ is compatible with the groupoid multiplication:
\[
(zp)\cdot (wq) = (zw)(p\cdot q)
\]
for all composable $p,q \in P$ and $z,w\in S^1$. 

Suppose  $(B,\gamma)$ defines a connective structure on $\calG$. Then $\pi_P^*\partial B = d\gamma$, where recall $\partial$ denotes the simplicial differential on $X^{[\bullet+1]}$. Letting $\delta$ denote the simplicial differential on $P_\bullet$, the simplicial manifold associated to $P\toto X$, we have $\delta B = d\gamma$.  Also, since the bundle isomorphism $\mu$ is connection-preserving, we see that $\delta \gamma = 0$. In other words, $\gamma$ is \emph{multiplicative}: $m^*\gamma = \pr_1^*\gamma+ \pr_2^*\gamma$.

\begin{remark} \label{r:differentialnotation}
(Notation) Throughout the paper, we shall use $\partial$ and $\delta$  to distinguish between the simplicial maps of $X^{[\bullet+1]}$ and $P_\bullet$, respectively, as in the preceding paragraph.
\eoe \end{remark}

\subsection{Multiplicative vector fields on bundle gerbes} \label{ss:mult}
After a brief review of multiplicative vector fields for Lie groupoids, this section studies multiplicative vector fields on bundle gerbes, establishing some basic properties used in the rest of the paper.  The examples included are meant to illustrate how these concepts are encoded for bundle gerbes given in terms of \v{C}ech data, which turn out to coincide with the notion of infinitesimal symmetries of gerbes in \cite{collier2011infinitesimal}.

Recall from \cite{mackenzie1998classical} that a \emph{multiplicative vector field} on a Lie groupoid $\mathbf{G}=\{G_1\toto G_0\}$ is a functor $\mathbf{x}:\mathbf{G} \to T\mathbf{G}$ such that $\pi_{\mathbf{G}}\circ \mathbf{x} = \id_{\mathbf{G}}$, where $\pi_{\mathbf{G}}:T\mathbf{G} \to \mathbf{G}$ denotes the tangent bundle projection. Such a functor $\mathbf{x}$  therefore consists of a pair of vector fields $(\mathbf{x}_0,\mathbf{x}_1) \in \mathfrak{X}(G_0) \times \mathfrak{X}(G_1)$ that are  compatible with units and the groupoid multiplications on $\mathbf{G}$ and $T\mathbf{G}$.  The collection of multiplicative vector fields  on $\mathbf{G}$ are the objects of a category $\mathbb{X}(\mathbf{G})$, whose morphisms are natural transformations $\eta: \mathbf{x}\Rightarrow \mathbf{x}'$ satisfying $\pi_\mathbf{G}(\eta(x))=\epsilon(x)$ for every object $x\in G_0$, where $\epsilon:G_0\to G_1$ denotes the unit map  \cite{hepworth2009vector}. 

Let us now specialize to the Lie groupoid $\mathbf{P}=\{P\toto X\}$ associated to a bundle gerbe $(X,P,\mu)$.
Consider a multiplicative vector field $({\mathbf{x}},{\mathbf{p}})$ on $\mathbf{P}$. Since $({\mathbf{x}},{\mathbf{p}})$ defines a functor, 
we have that ${\mathbf{p}} \sim_{\pi_j} {\mathbf{x}}$ for $j=0,1$, and hence $d\pi({\mathbf{x}}(x)) = d\pi({\mathbf{x}}(y))$ for all $(x,y) \in X^{[2]}$. Hence ${\mathbf{x}}$ descends to a vector field $\xi$ on $M$.

Moreover, $({\mathbf{x}},{\mathbf{x}})$ defines a vector field on $X^{[2]}$, and similarly on $X^{[n]}$ for any $n>1$; denote such vector fields by ${\mathbf{x}}^{[n]}$. In particular, we have that ${\mathbf{p}} \sim_{\pi_P} {\mathbf{x}}^{[2]}$. Similarly, there is a naturally defined vector field $({\mathbf{p}},{\mathbf{p}})$ on $P\times_X P$ (denoted by ${\mathbf{p}}\ast {\mathbf{p}}$ in \cite{mackenzie1998classical}) such that $({\mathbf{p}},{\mathbf{p}}) \sim_m {\mathbf{p}}$ and hence $({\mathbf{p}},{\mathbf{p}}) \sim_{\delta_j} {\mathbf{p}}$ for $j=0,1,2$.

\begin{proposition} \label{p:multvf}
Let $\calG=(X,P,\mu)$ be a bundle gerbe, and let $({\mathbf{x}},{\mathbf{p}})$ be a multiplicative vector field on  the corresponding Lie groupoid $P\toto X$.
\begin{enumerate}
	\item The vector field ${\mathbf{p}}$ is $S^1$-invariant.
	\item ${\mathbf{p}} \sim_{\pi_j} {\mathbf{x}}$ for $j=0,1$.
	\item ${\mathbf{p}}_2\otimes {\mathbf{p}}_0 \sim_\mu {\mathbf{p}}_1$, where ${\mathbf{p}}_j$ denotes the vector field on $\partial_j^*P$ induced by $({\mathbf{x}}^{[3]},{\mathbf{p}})$. 
\end{enumerate}
Moreover, any pair of vector fields $({\mathbf{x}},{\mathbf{p}}) \in \mathfrak{X}(X)\times \mathfrak{X}(P)$ satisfying the above conditions is a multiplicative vector field on $P\toto X$.
\end{proposition}
\begin{proof}
To show (1), we verify that $[\genz,{\mathbf{p}}]=0$. Choose a connection $\gamma$ as in Definition \ref{d:connex}. Since $[\genz,{\mathbf{p}}]$ is vertical, it suffices to show $\iota_{[\genz,{\mathbf{p}}]} \gamma =0$.
Note that since $\delta \gamma = 0$, it follows that $\delta \iota_{{\mathbf{p}}} \gamma =0$ and since $\mathbf{P}$ is proper there exists $g:X\to \RR$ such that $\delta g= \iota_{{\mathbf{p}}}\gamma$, which is $S^1$-invariant. Hence $\iota_{[\genz,{\mathbf{p}}]} \gamma  = L_{\genz} \iota_{{\mathbf{p}}}\gamma - \iota_{{\mathbf{p}}}L_{\genz} \gamma = L_{\genz} \delta g =0$, as required.

Statement (2) has already been established, 
and (3) is a  consequence of (1) and the discussion preceding the proposition.
To verify the last statement, observe that conditions (2) and (3) show that $({\mathbf{x}},{\mathbf{p}})$ defines a functor, provided it respects the unit maps. It thus suffices to check that $d\epsilon ({\mathbf{x}})(x) = {\mathbf{p}}(\epsilon(x))$, which follows immediately from (2) and the observation that $\epsilon$ is a section of $\pi_j$ (for $j=0$ or $1$).
\end{proof}

\begin{example} \label{eg:collier:symmetry}
Consider a multiplicative vector field $({\mathbf{x}},{\mathbf{p}})$ on a  bundle gerbe $P$ given in terms of \v{C}ech data, as in Example \ref{eg:cechdata}.  In this case, we have that ${\mathbf{x}}$ is specified by a collection of local vector fields $\xi_i$ on $U_i$ and by the paragraph preceding Proposition \ref{p:multvf}, $\xi_i\big|_{U_{ij}} = \xi_j\big|_{U_{ij}}$ for all $i,j$ and thus the local vector fields are restrictions of a vector field $\xi$ on $M$.  Denote the restriction of $\xi$ to $U_{i_1,i_2,\ldots}$ by $\xi_{i_1,i_2,\ldots}$.

The vector field ${\mathbf{p}}$ must take the form $(\xi_{ij}(x),f_{ij}(x,z))$, where  $f_{ij}(x,z) \in \RR$ is the component in the $S^1$ direction (where we implicitly trivialize the tangent bundle of $S^1$). By Proposition \ref{p:multvf} (1), the functions $f_{ij}$ are independent of $z\in S^1$; hence we abuse notation and write $f_{ij}:U_{ij} \to \RR$.
The composition
\[
\xymatrix@C=5em{
\xi_i(x) & \ar@/_1.5pc/[l]_-{(\xi_{ij},f_{ij}(x))} \xi_j(x) &  \ar@/_1.5pc/[l]_-{(\xi_{jk},f_{jk}(x))} \xi_k(x)
}
\]
in the tangent groupoid $TP\toto TX$ is  $(\xi_{jk},f_{ik}(x))$, as ${\mathbf{p}}$ is multiplicative. By definition, this composition is obtained by applying the tangent map of the gerbe multiplication.  Therefore,
\begin{equation} \label{eq:collier_objects}
f_{ik} = g_{ijk}^{-1} dg_{ijk}(\xi) +  f_{ij} +  f_{jk},
\end{equation}
on $U_{ijk}$. 

In \cite{collier2011infinitesimal}, Collier studies the infinitesimal symmetries of an $S^1$-gerbe considered as a stack over $M$.  In this context, the author defines a category $\mathcal{L}_{g_{ijk}}$ of \emph{infinitesimal symmetries} of the gerbe with \v{C}ech cocyle $g_{ijk}$. The objects of this category are given by pairs $(\xi,\{-f_{ij}\})$ consisting of a vector field $\xi$ on $M$ and a collection of functions $f_{ij}$ satisfying \eqref{eq:collier_objects}.
\eoe
\end{example}

In \cite{berwick2016lie}, the authors show that the category $\mathbb{X}(\mathbf{G})$ of multiplicative vector fields on a Lie groupoid $\mathbf{G}$ is a category internal to the category of vector spaces (i.e., a 2-vector space in the language of \cite{baez2004higher}). In fact, they show it is a Lie 2-algebra---we shall return to this in Section \ref{s:lie2}.  Though we are mainly interested in the case of Lie groupoids $P\toto X$ arising from bundle gerbes, there is no added difficulty in describing the morphisms in the general case, as we recall next, following \cite{berwick2016lie}.

For a general Lie groupoid $\mathbf{G}$, morphisms of multiplicative vector fields are given by sections of the Lie algebroid $A\to G_0$ of $\mathbf{G}$. (We shall implicitly make use of the equivalence of categories between 2-term chain complexes and 2-vector spaces---see \cite[Remark 2.9]{berwick2016lie} and \cite{baez2004higher}.)  To review some of the details, recall that the Lie algebroid $A=\ker ds \big|_{G_0}$, with anchor $dt:A\to TG_0$.  

A section $a\in \Gamma(A)$ gives rise to a morphism of multiplicative vector fields as follows. Let ${\mathbf{a}} = dt(a)$ and ${\bar{\mathbf{a}}} = \overrightarrow{a} + \overleftarrow{a}$, where 
\[
\overrightarrow{a}(g)=dR_g(a(t(g))) \quad \text{and} \quad \overleftarrow{a}(g) = d(L_g \circ i)(a(s(g))). 
\]
Here, $L_g$ and $R_g$ denote left and right multiplication, respectively,  by $g\in G_1$, and $i:G_1\to G_1$ denotes inversion.  It follows that $({\mathbf{a}},{\bar{\mathbf{a}}})$ is a multiplicative vector field \cite[Example 3.4]{mackenzie1998classical}, and we may view $a$ as a morphism  $a:({\mathbf{x}},{\mathbf{p}}) \to ({\mathbf{x}'},{\mathbf{p}'})$  whenever $({\mathbf{x}'},{\mathbf{p}'})-({\mathbf{x}},{\mathbf{p}}) =({\mathbf{a}},{\bar{\mathbf{a}}})$, with addition as composition.

Below we record some additional facts about multiplicative vector bundles for the case of bundle gerbes that we will use in Section \ref{s:lie2}.  In particular, we see from Proposition \ref{p:properties} (2) that the multiplicative vector fields 
isomorphic to the zero vector
are lifts of the zero vector on $M$.

\begin{lemma} \label{l:v_a}
Let $(X,P,\mu)$ be a bundle gerbe over $M$ with connection $(B,\gamma)$.
Let $a\in \Gamma(A_P)$, where $A_P \to X$ denotes the Lie algebroid of the corresponding Lie groupoid $P\toto X$. Set $\mathsf{v}_a = \epsilon^*\iota_{\overrightarrow{a}} \gamma$. Then $\iota_{{\bar{\mathbf{a}}}} \gamma =- \delta \mathsf{v}_a$.
\end{lemma}
\begin{proof}
The proof is a direct computation that uses the multiplicativity of the connection form $\gamma$. Indeed, for $p\in P$, we have
\begin{align*}
(\iota_{\overrightarrow{a}} \gamma)(p) &= \gamma_p dR_p (a(t(p))) \\
	&=  \gamma_{\epsilon(t(p))p} dR_p (a( t(p))) \\
	&= \gamma_{\epsilon(t(p))} (a( t(p))) \\
	&= \mathsf{v}_a(( t(p))).
\end{align*}
Similarly, $(\iota_{\overleftarrow{a}} \gamma)(p) = -  \mathsf{v}_a(s(p))$, as required.
\end{proof}

\begin{proposition} \label{p:properties}
Let $(X,P,\mu)$ be a bundle gerbe over $M$ with connection $(B,\gamma)$, and
  let $A_P$ denote the Lie algebroid of $\mathbf{P}$.
\begin{enumerate}
\item For any $g:X\to \RR$, $a=g\genz\big|_X$ defines a section of $A_P$ with ${\bar{\mathbf{a}}} = (\delta g) \genz$.
\item Let $({\mathbf{x}},{\mathbf{p}})$ be a multiplicative vector field. If ${\mathbf{x}} \sim_{\pi} 0$, then there exists a section $a\in \Gamma(A_P)$ with $({\mathbf{x}},{\mathbf{p}})=({\mathbf{a}},{\bar{\mathbf{a}}})$.
\end{enumerate}
\end{proposition}
\begin{proof}  
The first item is straightforward. To verify (2), it suffices to consider the case that ${\mathbf{p}}$ is horizontal, which is done next. Let  $a= \mathrm{Lift}_\gamma ({\mathbf{x}},0)\big|_X$ denote the restriction of the horizontal lift of $({\mathbf{x}},0)$ on $X^{[2]}$ to $P$. Since ${\mathbf{a}}={\mathbf{x}}$, it remains to check ${\bar{\mathbf{a}}} = {\mathbf{p}}$. By Lemma \ref{l:v_a}, ${\bar{\mathbf{a}}}$ is horizontal, and a direct calculation shows that ${\bar{\mathbf{a}}} \sim_{\pi_P} ({\mathbf{x}},{\mathbf{x}})$. Therefore, ${\bar{\mathbf{a}}}={\mathbf{p}}$, as required.
\end{proof}

\begin{example} \label{eg:collier:morphisms}
Consider, as in Examples \ref{eg:cechdata} and \ref{eg:collier:symmetry} , a bundle gerbe $P$ given by \v{C}ech data. 
The Lie algebroid $A_P$ in this case consists of vectors in $\bigsqcup TU_{i} \times TS^1$ of the form $(0,u)$. Therefore, we may identify a section $a\in \Gamma(A_P)$ with a collection of functions $a=\{u_i: U_i \to \RR\}$ that record the vertical components. Hence ${\mathbf{a}}=0$, and after identifying vertical vectors with functions, ${\bar{\mathbf{a}}} = \{ u_j-u_i\}$.  These are the morphisms of the category $\mathcal{L}_{g_{ijk}}$ in \cite{collier2011infinitesimal}. 
\eoe
\end{example}

\subsection{Multiplicative vector fields on bundle gerbes with connective structure} \label{ss:connex_multvf}

This section studies the interaction of multiplicative vector fields on bundle gerbes  with a connective structure. The section begins with Proposition \ref{p:curv_measures_bracket_preserve_failure}, which shows that the connective structure measures the failure of the horizontal lift to preserve brackets, and proceeds to introduce the category of connection-preserving multiplicative vector fields on bundle gerbes in Definition \ref{d:connexpresmultivf}. We continue the examples of the previous section to concretely illustrate the case of bundle gerbes given in terms of \v{C}ech data, recovering the connection-preserving infinitesimal symmetries of \cite{collier2011infinitesimal}.

Let $\calG=(X,P,\mu)$ be a bundle gerbe over $M$ with connective structure $(B,\gamma)$. In this section, we introduce the category of multiplicative vector fields on the corresponding groupoid $P\toto X$ that preserve the connective structure.

We begin with the following Proposition, which gives an interpretation of what the connective structure $(B,\gamma)$ on $\calG$ measures in terms of multiplicative vector fields, analogous to the interpretation given in \cite[Section 9]{collier2011infinitesimal}.

\begin{proposition} \label{p:curv_measures_bracket_preserve_failure}
Let $\calG=(X,P,\mu)$ be a bundle gerbe over $M$ with connective structure $(B,\gamma)$.  Let ${\mathbf{x}}$ and ${\mathbf{z}}$ be vector fields on $X$ that are lifts of vector fields $\xi$ and $\zeta$ on $M$.
\begin{enumerate}
	\item 
	Let  $\mathrm{Lift}_\gamma {\mathbf{x}}^{[2]}$ be the horizontal lift of ${\mathbf{x}}^{[2]}$ to $P$ with respect to $\gamma$. Then $({\mathbf{x}},\mathrm{Lift}_\gamma {\mathbf{x}}^{[2]})$ is a multiplicative vector field.
	\item 
	Let $g=\iota_{{\mathbf{x}}}\iota_{{\mathbf{z}}} B$ and $a = g\genz \big|_X \in \Gamma(A_P)$. Then $a$ defines an isomorphism 
	\[
	([{\mathbf{x}},{\mathbf{z}}],[\mathrm{Lift}_\gamma {\mathbf{x}}^{[2]},\mathrm{Lift}_\gamma {\mathbf{z}}^{[2]}])
	\longrightarrow ([{\mathbf{x}},{\mathbf{z}}],\mathrm{Lift}_\gamma [{\mathbf{x}},{\mathbf{z}}]^{[2]})
	\] 
	of multiplicative vector fields.
\end{enumerate}
\end{proposition}
\begin{proof}
To verify (1), observe that the first two items of Proposition \ref{p:multvf} are easily checked for  $({\mathbf{x}},\mathrm{Lift}_\gamma {\mathbf{x}}^{[2]})$. It remains to check compatibility with the groupoid multiplication. This follows from the observation that the product of horizontal vectors in $TP$ is horizontal, since the connection $\gamma$ is multiplicative.

To verify (2), we note 
\begin{align*}
\delta g &= \delta \iota_{{\mathbf{x}}}\iota_{{\mathbf{z}}} B \\
	&=  \iota_{\mathrm{Lift}_\gamma {\mathbf{x}}^{[2]}}\, \iota_{\mathrm{Lift}_\gamma{\mathbf{z}}^{[2]}} \delta B \\
	&= \iota_{\mathrm{Lift}_\gamma {\mathbf{x}}^{[2]}}\, \iota_{\mathrm{Lift}_\gamma{\mathbf{z}}^{[2]}} d\gamma \\ 
	&= \iota_{[\mathrm{Lift}_\gamma {\mathbf{x}}^{[2]},\mathrm{Lift}_\gamma {\mathbf{z}}^{[2]}]} \gamma
\end{align*}
Therefore, by Proposition \ref{p:properties} (1), $a \in \Gamma(A_P)$ satisfies
\[
({\mathbf{a}},{\bar{\mathbf{a}}})=(0, [\mathrm{Lift}_\gamma {\mathbf{x}}^{[2]},\mathrm{Lift}_\gamma {\mathbf{z}}^{[2]}]- \mathrm{Lift}_\gamma [{\mathbf{x}},{\mathbf{z}}]^{[2]}),
\]
as required.
\end{proof}

Let $({\mathbf{x}},{\mathbf{p}})$ be a multiplicative vector field on a bundle gerbe $\calG = (X,P,\mu)$ over $M$, and let $\xi$ denote the vector field on $M$ to which $\mathbf{x}$ descends. Suppose $\calG$ is given a connective structure $(B,\gamma)$ and that $\xi$ preserves the 3-curvature $\chi \in \Omega^3(M)$ of the connective structure---that is, $L_\xi \chi =0$. It follows that 
\[
dL_{\mathbf{x}} B = 0, \quad \delta L_{\mathbf{x}} B = dL_{\mathbf{p}} \gamma, \quad \text{and} \quad \delta L_{\mathbf{p}} \gamma=0.
\]
That is, a necessary condition on $(\mathbf{x},\mathbf{p})$ for $\xi$ to preserve the 3-curvature of a connective structure is that $L_\mathbf{x} B + L_\mathbf{p} \gamma \in \Omega^2(X)\oplus \Omega^1(P) $ be closed in the Bott-Shulman-Stasheff double complex $\Omega^*(P_\bullet)$ associated to the Lie groupoid $P\toto X$.  In Definition \ref{d:connexpresmultivf} below, we say that $({\mathbf{x}},{\mathbf{p}})$ preserves the connective structure $(B,\gamma)$ whenever $L_\mathbf{x} B + L_\mathbf{p} \gamma$ is exact.

 To further motivate  Definition \ref{d:connexpresmultivf}, suppose the vector fields $\xi$, $\mathbf{x}$, and  $\mathbf{p}$ have global flows $F$, $F_0$ and $\widetilde{F}_1$, respectively (where we suppress the time variable).  Since these flows are compatible with the groupoid structure, $(F_0,\widetilde{F}_1): \mathbf{P} \to \mathbf{P}$ defines a Lie groupoid isomorphism; however, in light of the definition of isomorphism in the bicategory of bundle gerbes over $M$ used here (Definition \ref{d:gerbe1arrow}), we ought to consider instead the more flexible notion of a \emph{Morita morphism} covering $F$, that is, an isomorphism  $Q: \calG \dashto F^*\calG$ of bundle gerbes over $M$.

To that end, we begin by introducing some notation used in the description of $F^*\calG$, as in the diagram below.
\[
\xymatrix{
{\hat{F}^{[2]*}}P \ar[r] \ar[d] & P \ar[d] \\
F^*{X^{[2]}} \ar[r]^-{\hat{F}^{[2]}} \ar@<-.7ex>[d]\ar@<.7ex>[d] & X^{[2]} \ar@<-.7ex>[d]\ar@<.7ex>[d] \\
F^*X \ar[r]^{\hat{F}} \ar[d] & X \ar[d]^-\pi \\
M \ar[r]^-F & M
}
\]
Consider the common refinement  
$\xymatrix{X & \ar@{=}[l] X \ar[r]^-{(\pi,F_0)} &F^*X}$, and take $Q=X\times S^1$. The Lie groupoid isomorphism $(F_0, \widetilde{F}_1)$ canonically defines an isomorphism of $S^1$-bundles $\varphi: (\pi,F_0)^*({\hat{F}^{[2]*}}P) \otimes \partial_0^*Q \to \partial_1^*Q \otimes P$, so that $(Q,\varphi)$ is the desired isomorphism of bundle gerbes. 
\\
\\
To refine $Q$ to an isomorphism of bundle gerbes with connective structure, let 
 $\gamma_Q = \pr_1^* \alpha + \pr_2^*\, d\theta$, for some $\alpha\in \Omega^1(X)$. Then $(Q,\gamma)$ will define an isomorphism $(\calG,B,\gamma) \dashto F^*(\calG, B, \gamma)$ whenever 
 \[
{F}_0^*B - B = d\alpha \quad \text{ and } \quad \widetilde{F}_1^*\gamma -\gamma = \delta \alpha.
\]
Hence, infinitesimally, we arrive at the following definition.

\begin{definition} \label{d:connexpresmultivf}
Let $({\mathbf{x}},{\mathbf{p}})$ be a multiplicative vector field on a bundle gerbe $\calG = (X,P,\mu)$ over $M$ with connective structure $(B,\gamma)$. We say  $({\mathbf{x}},{\mathbf{p}})$ \emph{preserves $(B,\gamma)$} if there exists a 1-form $\alpha \in \Omega^1(X)$ such that 
\begin{equation} \label{e:preserve}
L_{{\mathbf{x}}} B = d\alpha \quad \text{ and } \quad L_{{\mathbf{p}}} \gamma = \delta \alpha.
\end{equation}
\end{definition}

\begin{remark} \label{r:connexpresmultivf}
One could also consider multiplicative vector fields $({\mathbf{x}},{\mathbf{p}})$ on a bundle gerbe $\calG$ that preserve a connection $\gamma$ (without necessarily preserving a curving). In this case,  $\alpha$ is required to satisfy only the  condition $ L_{{\mathbf{p}}} \gamma = \delta \alpha$ from \eqref{e:preserve}.
\eoe \end{remark}

\begin{proposition} \label{p:connex_pres_subcat}
Let $\calG = (X,P,\mu)$ be a bundle gerbe over $M$ with connective structure $(B,\gamma)$, and let $a\in \Gamma(A_P)$. The multiplicative vector field $({\mathbf{a}},{\bar{\mathbf{a}}})$ preserves $(B,\gamma)$.
\end{proposition}
\begin{proof}
Let $\alpha = \iota_{{\mathbf{a}}} B - d\mathsf{v}_a$, and use Lemma \ref{l:v_a} and the equality $\delta B=d\gamma$ to check that
$
L_{{\mathbf{a}}} B = d\alpha$,  and $L_{{\bar{\mathbf{a}}}} \gamma = \delta \alpha$.
\end{proof}

\begin{corollary} \label{c:subcat}
Let $(X,P,\mu)$ be a bundle gerbe over $M$ with connective structure $(B,\gamma)$. 
There is a category $\mathbb{X}(\mathbf{P};B,\gamma)$ whose objects are pairs $({\mathbf{x}},{\mathbf{p}}; \alpha) \in \mathbb{X}(\mathbf{P})_0 \times \Omega^1(X) $  satisfying equation \eqref{e:preserve}, and whose morphisms are sections $a\in \Gamma(A_P)$: 
\[
a:({\mathbf{x}},{\mathbf{p}};\alpha) \to ({\mathbf{x}'},{\mathbf{p}'};\alpha'), \quad \text{iff} \quad ({\mathbf{x}'},{\mathbf{p}'};\alpha')-({\mathbf{x}},{\mathbf{p}};\alpha) =({\mathbf{a}},{\bar{\mathbf{a}}},\iota_{{\mathbf{a}}}B - d\mathsf{v}_a).
\]
Moreover, the forgetful functor  $\mathbb{X}(\mathbf{P};B,\gamma) \longrightarrow \mathbb{X}(\mathbf{P})$ is full and faithful.
\end{corollary}

\begin{example} \label{eg:multivf_trivgerbe}
Consider the trivial bundle gerbe $\calI_\omega$ with trivial connection and curving $\omega$. A multiplicative vector field is simply a vector field on $M$. Indeed, for a multiplicative vector field, the  component on $P=M\times S^1$ must be the trivial lift of a vector field $\xi$ on $M$, since its vertical component would (by virtue of being multiplicative) define a continuous homomorphism $S^1\to \RR$, which must be trivial. 
Therefore, a multiplicative vector field $(\xi,\xi+0)$ on $P\toto M$ preserves the connective structure if and only $L_\xi \omega$ is exact.
\eoe
\end{example}

\begin{example} \label{eg:collier:subcat}
Continuing from Examples \ref{eg:collier:symmetry} and \ref{eg:collier:morphisms}, let $(\xi_{ij},f_{ij})$ be a multiplicative vector field on the bundle gerbe $P$ given in terms of \v{C}ech data. Let $(\{B_i\},\{A_{ij}\})$ be a collection of 2-forms on the $U_i$ and 1-forms on the $U_{ij}$ that specify a connective structure on $P$, and suppose that $(\xi_{ij},f_{ij})$ preserves the connective structure. Then there exists a collection of 1-forms $a_i \in \Omega^1(U_i)$ satisfying 
\begin{align*}
\pr^*(a_j- a_i)&= L_{\xi_{ij}+\pr^*f_{ij} \genz} (\pr^*A_{ij} + d\theta) \\
	&= \pr^*L_{\xi_{ij}} A_{ij} + L_{\pr^*f_{ij} \genz}  d\theta \\
	&= \pr^*(L_{\xi_{ij}} A_{ij} + d f) 
\end{align*}
on $U_{ij}$.
That is, 
\begin{equation} \label{e:collier_connexpres}
a_j- a_i = L_{\xi_{ij}} A_{ij} + d f.
\end{equation}

In \cite{collier2011infinitesimal}, Collier defines the category $\mathcal{L}_{(g_{ijk},A_{ij})}$ whose objects $(\xi,\{-f_{ij}\}, \{-a_i\})$  are triples  consisting of a vector field $\xi$ on $M$, functions $f_{ij}$ satisfying \eqref{eq:collier_objects}, and 1-forms $a_i$ satisfying \eqref{e:collier_connexpres}, as well as a (full) subcategory $\mathcal{L}_{(g_{ijk},B_i,A_{ij})}$ whose objects satisfy the additional constraint $L_\xi B_i = da_i$ (i.e., infinitesimal symmetries that preserve the connection and the curving). As in Example \ref{eg:collier:morphisms}, the morphisms are given by collections of functions $\{u_i:U_i \to \RR\}$, which, as noted in that Example, correspond to sections $a$ of the underlying Lie algebroid. Therefore, $\{u_i\}$ is a morphism $(\xi,\{-f_{ij}\}, \{-a_i\}) \longrightarrow  (\xi,\{-g_{ij}\}, \{-b_i\})$ if and only if $f_{ij}-g_{ij} = u_j-u_i$ and $b_i-a_i=du_i$, and we recover the category of connection-preserving infinitesimal symmetries from \emph{op.\ cit.} 
 \eoe
\end{example}

\subsection{Multiplicative vector fields on bundle gerbes as infinitesimal symmetries of the underlying gerbe} \label{ss:infsym}

The purpose of this section is to relate multiplicative vector fields on bundle gerbes to Collier's notion of infinitesimal symmetries of $S^1$-gerbes  viewed as stacks  \cite{collier2011infinitesimal}. Recall from Examples \ref{eg:collier:symmetry} and \ref{eg:collier:morphisms} that the category of multiplicative vector fields on a bundle gerbe $\mathcal{G}_{g_{ijk}}$ over a manifold $M$, given in terms of  a \v{C}ech cocycle $\{ g_{ijk} \}$ associated to an open cover $\{U_i\}$ of $M$, 
coincides with the category $\mathcal{L}_{g_{ijk}}$ of \emph{infinitesimal symmetries} of the gerbe as defined by Collier in \emph{op.\ cit.\ }.    In this section, we verify in Theorem \ref{t:equiv-mvf-col} that the category of multiplicative vector fields $\mathbb{X}(\mathbf{P})$ on a bundle gerbe $\calG=(X,P,\mu)$ is equivalent to the category $\mathcal{L}_{\calG}$ in \emph{op.\ cit.\ }of infinitesimal symmetries of the underlying gerbe (i.e., viewed as a \emph{stack} over $M$). 
A similar discussion in Example \ref{eg:collier:subcat} considers infinitesimal symmetries that preserve the connective structure of the gerbe, which is correspondingly generalized in Theorem \ref{t:equiv-mvf-col-withconnex}.  

The content of this section is not used elsewhere in the paper. Indeed, a main objective of the paper is to show how multiplicative vector fields encode infinitesimal symmetries without relying on the perspective or machinery of stacks.
Therefore, in the interest of brevity, this section  assumes a familiarity with the \emph{presheaf of groupoids} (i.e., stack) viewpoint of $S^1$-gerbes (e.g., as in Brylinski's book \cite{brylinski2007loop}). We shall follow Collier's notation closely, and review some  elements in  \cite{collier2011infinitesimal} to state our results precisely in that context.

To begin, recall how a bundle gerbe $\calG=(X,P,\mu)$ over $M$ gives rise to a gerbe, which by abuse of notation will also be denoted by $\calG$.  For each open set $U\subset M$, $\calG(U)$ is the category of local trivializations of the bundle gerbe restricted to $U$. More precisely, objects of $\calG(U)$ are pairs $(Q,\phi)$ consisting of an $S^1$-bundle $\pi_Q:Q \to X_U = \pi^{-1}(U)$ together with a bundle isomorphism $\phi:\partial Q = \partial_1^* Q^* \otimes \partial_0^* Q \to P_U =P\big|_{X_U^{[2]}}$, satisfying a compatibility condition with bundle gerbe product $\mu$ (\emph{cf}.\ Definition \ref{d:gerbe1arrow}). A morphism $(Q,\phi) \to (Q',\phi')$ is given by a bundle isomorphism $u:Q\to Q'$ that is compatible with $\phi$ and $\phi'$, $\phi \circ \partial u = \phi'$. The resulting presheaf of groupoids $\calG$ is seen to be a \emph{stack}; moreover, it naturally has \emph{band} equal to the sheaf of $S^1$-valued functions on $M$. That is, $\calG$ is an $S^1$-(banded) gerbe (see \cite{stevenson2000geometry} for details).

\subsubsection*{Infinitesimal symmetries of $S^1$-gerbes}

Let $\underline{\RR}_M$ and $\underline{S}^1_M$ denote the sheaves of smooth real and circle valued functions on $M$, respectively. For a vector field $\xi$ on $M$ and a function $g:M\to S^1$, let $\iota_{\xi} d\, \log(g) = \iota_{\xi} g^{-1}dg$, which defines a morphism of sheaves $\iota_{\xi} d\, \log: \underline{S}^1_M \to \underline{\RR}_M$.

Given  a vector field $\xi$ on $M$, Collier defines a category $\mathcal{L}_{\calG}(\xi)$ whose objects are \emph{infinitesimal symmetries} of an $S^1$-gerbe $\calG$ that \emph{lift} $\xi$.  Specifically, an object in $\mathcal{L}_{\calG}(\xi)$ is a morphism of gerbes
(i.e., a morphism of presheaves or a pseudo-natural transformation)
\[
F:\calG \to B\underline{\RR}_M,
\]
intertwining the morphisms of sheaves, $-\iota_{\xi} d\, \log$.  Here, $B\underline{\RR}_M$ denotes the trivial gerbe with band  $\underline{\RR}_M = C^\infty(M,\RR)$: given an open set $U\subset M$, $B\underline{\RR}_M(U)$ is the category of $\underline{\RR}_U$-torsors. In particular, we obtain a family of functors $F_U:\calG(U) \to B\underline{\RR}_M(U)$ (or simply $F$ when the underlying open set $U$ is understood) such that for 
 all objects $Q$ in $\calG(U)$,  $F(Q)$ is an $\underline{\RR}_U$-torsor; moreover, for
all morphisms $u:Q\to Q$, and all $g:U\to S^1$, $F(u): F(Q) \to F(Q)$ satisfies $F(u\cdot g) = F(u) - \iota_{\xi} g^{-1}dg$.  

A morphism $F\Rightarrow F'$ of infinitesimal symmetries lifting $\xi$ is a modification of the underlying pseudo-natural transformations. In particular, such a morphism is specified by a map of $\underline{\RR}_U$-torsors $T:F(Q) \to F'(Q)$ for every object $Q$ in $\calG(U)$.
Let $\mathcal{L}_{\calG}$ denote the category 
whose objects consist of the collection infinitesimal lifts $F$ of vector fields $\xi$ on $M$ as above, and whose morphisms are exactly those in $\mathcal{L}_{\calG}(\xi)$ described above  (i.e., there are no morphisms between infinitesimal lifts of different vector fields on $M$). Theorem \ref{t:equiv-mvf-col} below shows $\mathcal{L}_{\calG}$ is equivalent to the category of multiplicative vector fields $\mathbb{X}(\mathbf{P})$.

\begin{lemma} \label{l:algebroid}
Let $\pi:X\to M$ be a surjective submersion, let $\pi_Q:Q \to X$ be a principal $S^1$-bundle, and let $\partial Q = \partial_1^*Q^* \otimes \partial_0^* Q \to X^{[2]}$ and $(X,\partial Q, \mu)$ be the resulting canonically trivialized bundle gerbe. Let $a\in \Gamma(A_{\partial Q})$, and recall ${\mathbf{a}}=dt (a)$, and ${\bar{\mathbf{a}}} = \overrightarrow{a} + \overleftarrow{a}$. 
\begin{enumerate}
\item There is a unique vector field $\widehat{a} \in \mathfrak{X}(Q)^{S^1}$ such that the section $a$ is of the form 
$
a(x) = [({\mathbf{a}}(x),0,\widehat{a}(q),0)]$, where $\pi_Q(q)=x.
$

\item 
\(
{\bar{\mathbf{a}}}(x_1, x_2, q_1\otimes q_2) = [({\mathbf{a}}(x_1),{\mathbf{a}}(x_2),\widehat{a}(q_1),\widehat{a}(q_2))]
\) for all $(x_1,x_2,q_1\otimes q_2) \in \partial Q$.
\item Let $\gamma_Q$ denote a connection form on $Q\to X$, inducing a bundle gerbe connection $\partial \gamma_Q$ on $\partial Q$. Then $\pi_Q^*\mathsf{v}_a = - \iota_{\widehat{a}} \gamma_Q$.
\end{enumerate}
\end{lemma}
\begin{proof}
We begin by reviewing elementary aspects of the Lie algebroid $A_{\partial Q}$. For a point $\mathbf{q}=(x_1,x_2,q_1\otimes q_2) \in \partial Q$, where $(x_1,x_2)\in X^{[2]}$ and $\pi_Q(q_i)=x_i$, vectors in $T_{\mathbf{q}} \partial Q$ may be represented by equivalences classes $\mathbf{v}=[(u_1,u_2,v_1,v_2)]$ where $u_i \in T_{x_i} X$, $v_i \in T_{q_i} Q$ satisfy $d\pi(u_1)=d\pi(u_2)$ and $d\pi_Q(v_i) = u_i$ for $i=1,2$. 

For $\mathbf{v} \in \ker ds$, where $s(\mathbf{q})=x_2$ is the source map of $\partial Q \toto X$, then $u_2=0$. Thus there exists a unique representative for $\mathbf{v}$ of the form $\mathbf{v}=[(u_1,0,w,0)]$. Therefore, letting $\widetilde{s}:\partial_1^*Q^* \times_{X^{[2]}} \partial_0^* Q \to X$  and $\widetilde{\partial}_1:\partial_1^*Q^* \times_{X^{[2]}} \partial_0^* Q \to Q$ denote the maps $\widetilde{s}(x_1,x_2,q_1,q_2) = x_2$ and $\widetilde{\partial}_1(x_1,x_2,q_1,q_2)=q_2$, respectively, we see that the map 
\begin{align*}
(\ker d\widetilde{s} \cap \ker d\widetilde{\partial}_1)\big|_{Q} &\longrightarrow Q\times_X A_{\partial Q} \\
(x,x,q,q;u,0,w,0) & \longmapsto (q;x,x,q\otimes q; [u,0,w,0])
\end{align*}
is a diffeomorphism. 

To prove (1), let $a\in \Gamma(A_{\partial Q})$. We obtain the vector field $\widehat{a} \in \mathfrak{X}(Q)$ as the composition
\[
\xymatrix@C=9ex{
Q \ar[r]^-{(\id,a \circ\pi_Q)} & Q\times_X A_{\partial Q} \ar[r] & (\ker d\widetilde{s} \cap \ker d\widetilde{\partial}_1)\big|_{Q} \ar[r]^-{\pr} & TQ.
}
\]
The vector field $\widehat{a}$ is $S^1$-invariant: 
\[
 [({\mathbf{a}}(x),0,\widehat{a}(q\cdot z),0)] = [({\mathbf{a}}(x),0,\widehat{a}(q),0)] =  [({\mathbf{a}}(x),0,dR_z \, \widehat{a}(q),0)] .
\]
Item (2) is a straightforward calculation. Similarly, since $\overrightarrow{a}(x_1,x_2,q_1\otimes q_2) = [({\mathbf{a}}(x_1),0,\widehat{a}(q_1),0 )]$, and the induced connection on the bundle gerbe $\partial Q$ is $\partial \gamma_Q = \partial_1^*\gamma_Q^* \otimes \partial_0^*\gamma_Q$, we see that for $q\in Q$,
\begin{align*}
\mathsf{v}_a (\pi_Q(q))& = ( \partial \gamma_Q )_{(\pi_Q(q),\pi_Q(q);q\otimes q)} [({\mathbf{a}}(\pi_Q(q)),0,\widehat{a}(q),0 )] \\ 
&= -(\gamma_Q)_q(\widehat{a}(q)),
\end{align*}
as required.
\end{proof}

\begin{theorem} \label{t:equiv-mvf-col}
Let $\calG=(X,P,\mu)$ be a bundle gerbe over $M$. There is an equivalence of categories $\mathbb{X}(\mathbf{P}) \to \mathcal{L}_{\calG}$.
\end{theorem}
\begin{proof}
We construct the functor $\mathbb{X}(\mathbf{P}) \to \mathcal{L}_{\calG}$ directly and provide a sketch that it is an equivalence of categories. Let $\overline{\xi}=({\mathbf{x}},{\mathbf{p}})$ be a multiplicative vector field that descends to a vector field $\xi$ on $M$. Let $F_{\overline{\xi}}:\calG \to B\underline{\RR}_M$ be given as follows. To each object $(Q,\phi)$ in $\calG(U)$, let
\[
F_{\overline{\xi}}(Q) = \{\widehat{\xi} \in \mathfrak{X}(Q)^{S^1}\,  | \, \widehat{\xi} \sim_{\pi_Q} {\mathbf{x}}\big|_{X_U}, \, \text{and }\, d\phi(\delta\widehat{\xi}) = {\mathbf{p}}\big|_{P_U} \}
\]
be the corresponding $\underline{\RR}_U$-torsor. Here, using the $S^1$-invariance of $\widehat{\xi}$,   $\delta\widehat{\xi}$ 
denotes the vector field on $\partial_0^*Q \otimes \partial_1^*Q$ induced from $({\mathbf{x}},\widehat{\xi})$ on each factor $\partial_i^*Q$, $i=0,1$.

Let  $a\in \Gamma(A_P)$, and let $a_U \in \Gamma(A_{P_U})$ denotes its restriction to $X_U$.
Observe that $\phi:\partial Q \to P_U$ induces a Lie groupoid isomorphism 
\[
\xymatrix{
\partial_0^* Q \otimes \partial_1^*Q \ar[r]^-\phi  \ar@<-.7ex>[d]\ar@<.7ex>[d] & P_U \ar@<-.7ex>[d]\ar@<.7ex>[d] \\
X_U \ar@{=}[r] & X_U
}
\]
and hence we may view $a_U$ as a section of the Lie algebroid of $\partial Q \toto X_U$. 
The section $a_U$ gives rise to an $S^1$-invariant vector field $\widehat{a}$ on $Q$, as in Lemma \ref{l:algebroid}, which satisfies 
$\widehat{a} \sim_{\pi_Q} {\mathbf{a}}\big|_{X_U}, \, \text{and }\, d\phi(\delta\widehat{a}) = {\bar{\mathbf{a}}}\big|_{P_U}$.

Viewing $a:({\mathbf{x}},{\mathbf{p}}) \to ({\mathbf{x}'},{\mathbf{p}'})$ as a morphism in $\mathbb{X}(\mathbf{P})$, the vector fields $\widehat{a}$ from the previous paragraph give rise to a morphism $T_a:F_{\overline{\xi}} \Rightarrow F_{\overline{\zeta}}$. Specifically, for each $(Q,\phi) \in \calG(U)$, we obtain a map of $\underline{\RR}_U$-torsors
$T_a:F_{\overline{\xi}}(Q) \to F_{\overline{\zeta}}(Q)$, by setting
\(
T_a(\widehat{\xi}) = \widehat{\xi}+\widehat{a}.
\) Hence, we obtain a functor $\mathbb{X}(\mathbf{P}) \to \mathcal{L}_{\calG}$.

To show that the functor constructed above is an equivalence of categories, note first that, by \cite[Corollary 6.1.8 (2)]{collier2011infinitesimal}, it is immediate that the functor is essentially surjective.  Faithfulness follows from the uniqueness of the vector field $\widehat{a}$, as in Lemma \ref{l:algebroid}. To see that the functor is full, 
suppose we are given an equivalence of lifts $T:F_{\overline{\xi}}(Q) \to F_{\overline{\zeta}}(Q)$. There exists $a'\in \Gamma(A_P)$ with $(\mathbf{a}',\bar{\mathbf{a}}')=({\mathbf{x}'} - {\mathbf{x}},{\mathbf{p}'}-{\mathbf{p}})$, and the collection of such sections $a'$ is a torsor for $C^\infty(M)$ (i.e., 
sections of the isotropy Lie algebra)---this follows from Prop \ref{p:properties} (3) and (4)  and the  exact sequence \ref{eq:deltaexact}. For $\widehat{\xi} \in F_{\overline{\xi}}(Q)$,  $T(\widehat{\xi})-T_{a'}(\widehat{\xi})$ is an $S^1$-invariant lift of $0$ (restricted to $X_U$), which may thus identify with a (basic) function $f:U\to \RR$ that records its vertical component. Since 
 $d\phi( \delta (T(\widehat{\xi})-T_{a'}(\widehat{\xi})))=0$,
 $\delta f=0$, and we have that such functions patch together to give a global function $g:M\to \RR$. Letting $a=a' + (\pi^*g) \genz$, it is straightforward to check that $T=T_a$, as required.
\end{proof}

\subsubsection*{Infinitesimal symmetries of $S^1$-gerbes preserving a connective structure}

Recall that for an $S^1$-gerbe $\calG$, viewed as a presheaf of groupoids, a connection on $\calG$ is a morphism of gerbes $\calA:\calG \to B\underline{\Omega}^1_M$ intertwining the morphism of sheaves $-d\log: \underline{S}^1_M \to \Omega^1_M$. Here, $B\underline{\Omega}^1_M$ denotes the trivial gerbe with band $\underline{\Omega}^1_M$: specifically, given an  open set $U\subset M$, $\underline{\Omega}^1_M(U)$ is the category of $\underline{\Omega}^1_U$-torsors. In particular, we obtain a family of functors $\calA_U: \calG(U) \to B\underline{\Omega}^1_M(U)$ (or simply $\calA$ when the underlying open set $U$ is understood) such that for all objects $Q\in \calG(U)$, $\calA(Q)$ is a  $\underline{\Omega}^1_U$-torsor; moreover, for all morphisms $u:Q\to Q$, and all $g:U\to S^1$, $\calA(u): \calA(Q) \to \calA(Q)$ satisfies  $\calA(u\cdot g) = \calA(u) - d\log\, g$.  A \emph{curving}
$K$ for the connection $\calA$ assigns to each $Q\in \calG(U)$ and each $\gamma \in \calA(Q)$, a 2-form $K(\gamma)\in \Omega^2(U)$. This assignment is compatible with restrictions of open subsets, and morphisms $u:Q\to Q'$---that is,  $K((\calA)(u)(\gamma))=K(\gamma)$. Moreover, $K$ is $\Omega^1(U)$-equivariant: for all $\alpha \in \Omega^1(U)$,  $K(\gamma \cdot \alpha) = K(\gamma) + d\alpha$.  We will call the pair $(\calA,K)$ a \emph{connective structure} for $\calG$.

Observe that for an $S^1$-gerbe $\calG$ associated to a bundle gerbe $(X,P,\mu)$, a connective structure $(B,\gamma)$ for the bundle gerbe gives rise to a connective structure in the sense described above. Namely, for $U\subset M$ and $(Q,\phi)\in \calG(U)$, we set 
\[
\calA(Q) = \{ \text{connection 1-forms } \gamma_Q  \in \Omega^1(Q): \phi^*(\gamma|_{P_U}) = \partial \gamma_Q \},
\]
where $\partial \gamma_Q = \partial_1^*\gamma_Q^* \otimes \partial_0^* \gamma_Q$ is the induced connection on $\partial Q$.
The curving $K$ is then obtained as follows: observe that $\partial \mathrm{curv}(\gamma_Q) = \mathrm{curv}(\gamma|_{P_U}) =\partial B|_U$ and hence there exists a unique 2-form  $K(\gamma_Q)$ satisfying $\pi^*K(\gamma_Q) = \mathrm{curv}(\gamma_Q) - B|_U$.  Note that it follows that $dK(\gamma_Q) = -\chi|_U$. 

In \cite{collier2011infinitesimal, collier-PhDThesis2012},  Collier defines categories $\mathcal{L}_{(\calG,\calA)}$ and $\mathcal{L}_{(\calG,\calA,K)}$ of  infinitesimal symmetries of $\calG$  that preserve the connection and the connective structure, respectively. 
 Recall that an infinitesimal symmetry lifting $\xi \in \mathfrak{X}(M)$ that preserves the connection $\mathcal{A}$, also called a \emph{connective lift} of $\xi$, is a pair $(F,\Theta)$ consisting of an infinitesimal lift $F:\calG \to B\underline{\RR}_M$ in  $\mathcal{L}_{\calG}(\xi)$ and a (2-)morphism 
$\Theta$ for the diagram:
\[
\xymatrix{
\calG \ar[r]^-{\calA} \ar[d]_{F} \drtwocell\omit{<0>\Theta} & B\underline{\Omega}^1_M \ar[d]^{L_{-\xi}[\,]} \\
B\underline{\RR}_M \ar[r]^{-d[\,]} & B\underline{\Omega}^1_M
}
\]
That is, $\Theta$ is a modification for the two composite pseudo-natural transformations represented in the diagram. Recall that in this context, $-d[\,]:B\underline{\RR}_M \to B\underline{\Omega}^1_M$ may be described as follows: for each open set $U\subset M$
we obtain a family of functors $B\underline{\RR}_M(U) \to B\underline{\Omega}^1_M(U)$ associating to each $\underline{\RR}_U$-torsor, the $\underline{\Omega}^1_U$-torsor associated to it via $-d:C^\infty(U) \to \Omega^1(U)$. That is, for any open $V\subset U$ and any $C^\infty(V)$-torsor $A$ over $V$, $-d[A] = (A\times \Omega^1(V))/\sim$, where by definition $(a\cdot f,\mu)\sim (a,\mu-df)$ for any $f:V\to \RR$. The map $L_{-\xi}[\,]$ may be described similarly.

A connection-preserving multiplicative vector field $({\mathbf{x}},{\mathbf{p}}; \alpha) \in \mathbb{X}(\mathbf{P};B, \gamma)$ naturally gives rise to a connective lift $(F_{\overline{\xi}},\Theta_{(\overline{\xi},\alpha)})$ as follows. The infinitesimal symmetry $F_{\overline{\xi}}$ is as above.  To describe $\Theta_{(\overline{\xi},\alpha)}$, we define for each $U\subset M$ and $(Q,\phi)\in \calG(U)$, a map of $\Omega^1(U)$-torsors $\Theta_{(\overline{\xi},\alpha)}(Q):L_{-\xi}[\calA(Q)] \to -d[F_{\overline{\xi}}(Q)]$. 
 To that end, since $\delta \iota_{{\mathbf{p}}} \gamma =0$, there exists a function  $g \in C^\infty(X)$ satisfying $\iota_{{\mathbf{p}}} \gamma = \delta g$, and we may set 
\[
\widehat{\xi}_g= \mathrm{Lift}_{\gamma_Q}({\mathbf{x}}) + (\pi_Q^*g|_{X_U}) \genz \in F_{\overline{\xi}}(Q).
\]
Since $\delta(\alpha-\iota_{{\mathbf{x}}}\, B - dg) =0$,  there exists a unique 1-form $\varepsilon_g \in \Omega^1(M)$ such that $\pi^*\varepsilon_g = \alpha-\iota_{{\mathbf{x}}}\, B - dg$. Now for $[(\gamma_Q,\tau)] \in L_{-\xi}[\calA(Q)]$, let
\[
\Theta_{(\overline{\xi},\alpha)}(Q)[(\gamma_Q,\tau)] = [(\widehat{\xi}_g,\tau - \iota_{\xi} K(\gamma_Q) - \varepsilon_g)].
\]
Observe that the above map is independent of the choice of $g$: any other such choice $g'$ must be of the form $g'=g+\pi^*h$ for some $h\in C^\infty(M)$, which yields $\varepsilon_{g'}=\varepsilon_g - dh$, and $\widehat{\xi}_{g'} = \widehat{\xi}_g + (\pi^*_Q \pi^* h|_{X_U}) \genz$. Hence,
\(
(\widehat{\xi}_{g'},\tau - \iota_{\xi} K(\gamma_Q) - \varepsilon_{g'}) \sim (\widehat{\xi}_{g},\tau - \iota_{\xi} K(\gamma_Q) - \varepsilon_{g}).
\)

To check $\Theta_{(\overline{\xi},\alpha)}(Q)$ is well-defined, 
\begin{align*}
 \Theta_{(\overline{\xi},\alpha)}(Q)&[(\gamma_Q+\pi_Q^*\pi^*\nu,\tau+L_{\xi}\nu)] \\
 &=  [(\widehat{\xi}_g - \pi_Q^*\pi^*\iota_\xi \nu \genz, \tau +L_{\xi}\nu - \iota_{\xi} K(\gamma_Q+\pi_Q^*\pi^*\nu) - \varepsilon_g)] \\
 &=  [(\widehat{\xi}_g , \tau +L_{\xi}\nu - \iota_{\xi} (K(\gamma_Q) + d\nu )- \varepsilon_g -d\iota_{\xi} \nu)] \\
&= \Theta_{(\overline{\xi},\alpha)}(Q)[(\gamma_Q,\tau)].
\end{align*}

A connective lift $(F,\Theta)$ preserves the curving $K$ if for each $(Q,\phi) \in \calG(U)$ and each $\gamma \in \calA(Q)$, 
\[
d\Theta(Q)(\gamma) = -L_{\xi} K(\gamma).
\]
In this context, recall that if we write $\Theta(Q)[(\gamma,0)]=[(\widehat{\xi},\mu)]$, then $d\mu \in \Omega^2(U)$ is well-defined, and hence we set $d\Theta(Q)(\gamma) = d\mu$. 
Therefore, continuing from the above discussion, the connective lift $(F_{\overline{\xi}},\Theta_{(\overline{\xi},\alpha)})$ associated to $({\mathbf{x}},{\mathbf{p}}; \alpha) \in \mathbb{X}(\mathbf{P};B, \gamma)$  preserves the curving $K$ if 
$ d(- \iota_{\xi} K(\gamma_Q) - \varepsilon_{g}) = -L_{\xi}K(\gamma_Q) $, or equivalently $d\varepsilon_g~=~\iota_\xi dK(\gamma_Q)$. Since $dK(\gamma_Q) = -\chi|_U$, we see that $(F_{\overline{\xi}},\Theta_{(\overline{\xi},\alpha)})$ preserves the curving if $\xi$ is Hamiltonian, with Hamiltonian 1-form $\varepsilon_g$.

A morphism $(F,\Theta) \Rightarrow (F',\Theta')$ of connective lifts is a so-called \emph{connective equivalence}; i.e., a morphism of $F\Rightarrow F'$ of the underlying infinitesimal symmetries that is compatible with $\Theta$ and $\Theta'$. A morphism of connective lifts that preserve the curving is simply a morphism of the underlying connective lifts. (\emph{cf}.\  \cite[Definition 4.1.16]{collier-PhDThesis2012}). 

For a morphism $a:({\mathbf{x}},{\mathbf{p}};\alpha) \to ({\mathbf{x}'},{\mathbf{p}'};\alpha')$ of multiplicative vector fields preserving the connective structure, we may associate a morphism $T_a:F_{\overline{\xi}} \Rightarrow F_{\overline{\zeta}}$,
as in the proof of Theorem \ref{t:equiv-mvf-col}. It is straightforward to see that $T_a$ is compatible with $\Theta_{(\overline{\xi},\alpha)}$ and $\Theta_{(\overline{\zeta},\beta)}$. In this context, compatibility amounts to the commutativity of the following diagram,
\[
\xymatrix{
L_{-\xi}[\calA(Q)] \ar[r]^-{\Theta_{(\overline{\xi},\alpha)}} \ar[rd]_-{\Theta_{(\overline{\zeta},\beta)}} & -d[F_{\overline{\xi}}(Q)] \ar[d]^-{-d[T_a]} \\
& -d[F_{\overline{\zeta}}(Q)]
}
\]
Before this is verified, we first choose $g\in C^\infty(X)$ satisfying $\iota_{{\mathbf{p}}} \gamma = \delta g$, and set $h=g-\mathsf{v}_a$. Then, it is easy to see that $\varepsilon_g = \varepsilon_h \in \Omega^1(M)$ (as defined earlier). Moreover, by Lemma \ref{l:algebroid} (3), we have $\widehat{\zeta}_h = \widehat{\xi}_g+ \widehat{a}$. Therefore,
\begin{align*}
-d[T_a] \circ \Theta_{(\overline{\xi},\alpha)} [(\gamma_Q,\tau)] &=-d [T_a] [(\widehat{\xi}_g,\tau - \iota_{\xi} K(\gamma_Q) - \varepsilon_g)] \\
 &= [(T_a(\widehat{\xi}_g),\tau - \iota_{\xi} K(\gamma_Q) - \varepsilon_g)]  \\
 &= [(\widehat{\zeta}_h,\tau - \iota_{\xi} K(\gamma_Q) - \varepsilon_h)] \\
 &= \Theta_{(\overline{\zeta},\beta)} [(\gamma_Q,\tau)],
\end{align*}
as required.

Similar to Theorem \ref{t:equiv-mvf-col}, we see that the above discussion  describes a functor $\mathbb{X}(\mathbf{P}; \gamma) \to \mathcal{L}_{(\calG,\calA)}$, which is an equivalence of categories.

\begin{theorem} \label{t:equiv-mvf-col-withconnex}
Let $\calG=(X,P,\mu)$ be a bundle gerbe over $M$. There is an equivalence of categories $\mathbb{X}(\mathbf{P}; \gamma) \to \mathcal{L}_{(\calG,\calA)}$, which restricts to an equivalence of categories $\mathbb{X}(\mathbf{P}; B, \gamma) \to \mathcal{L}_{(\calG,\calA, K)}$.
\end{theorem}
\begin{proof}
That the functor $\mathbb{X}(\mathbf{P}; \gamma) \to \mathcal{L}_{(\calG,\calA)}$ is full and faithful is straightforward. It remains to show  it is essentially surjective. 

Let $(F,\Theta)$ be a connective lift of some vector field $\xi$ on $M$. By Theorem \ref{t:equiv-mvf-col}, there exists a multiplicative vector field $\overline{\xi}=({\mathbf{x}},{\mathbf{p}})$ and an isomorphism $S:F\Rightarrow F_{\overline{\xi}}$. Moreover, setting $\Theta' = (S * \id_{-d[\, ]}) \circ \Theta$, we  extend $F_{\overline{\xi}}$ to a connective lift $(F_{\overline{\xi}}, \Theta')$, which is  isomorphic to $(F,\Theta)$. 

Let $g \in C^{\infty}(X)$  satisfy $\delta g = \iota_{\mathbf{p}} \gamma$, and consider $\alpha_g = \iota_{{\mathbf{x}}} B + dg \in \Omega^1(X)$. It is easy to see that $L_{{\mathbf{p}}} \gamma = \delta \alpha_g$, and that the collection of $\alpha$ for which $({\mathbf{x}},{\mathbf{p}};\alpha) \in \mathbb{X}(\mathbf{P};\gamma)$ is a $\Omega^1(M)$-torsor. Since $(F_{\overline{\xi}},\Theta_{(\overline{\xi},\alpha_g)})$ is a connective lift extending $F_{\overline{\xi}}$, and the set of such connective lifts is also a $\Omega^1(M)$-torsor (by \cite[Corollary 10.46]{collier2011infinitesimal}), $(F_{\overline{\xi}}, \Theta')$ is isomorphic to some $(F_{\overline{\xi}},\Theta_{(\overline{\xi},\alpha)})$, as required.

The statement about restricting the functor follows from the observation that if there exists $(F,\Theta) \Rightarrow (F',\Theta')$, then $(F,\Theta)$ preserves the curving iff $(F',\Theta')$ does (\emph{cf}.\ \cite[Def.\ 4.1.16]{collier-PhDThesis2012}).
\end{proof}

\section{Lie 2-algebras and 2-plectic geometry} \label{s:lie2}

Following a review of the bicategory of Lie 2-algebras, we introduce the main examples of interest in this paper---namely, the Poisson Lie 2-algebra of a 2-plectic manifold and the connection-preserving multiplicative vector fields on (i.e., infinitesimal symmetries of) a bundle gerbe (Proposition \ref{p:multivfLie2}). In Theorem \ref{p:3curv_obs}, we show that the horizontal lift of vector fields determined by a connective structure fails to be a morphism of Lie 2-algebras, unless its 3-curvature vanishes.

\subsection{Review of Lie 2-algebras} \label{ss:Lie2defs}

In this Section we review Lie 2-algebras and  Noohi's bicategory $\LieButterfly$ of 2-term $L_\infty$-algebras, with butterflies as 1-arrows, and morphisms of butterflies as 2-arrows \cite{noohi2013integrating}. 
As shown in \emph{op.\ cit.}, $\LieButterfly$ is the localization of the bicategory    $\LieBC$ of 2-term $L_\infty$-algebras defined by Baez and Crans \cite{baez2004higher} with respect to quasi-isomorphisms.  
We begin with a brief review of $\LieBC$, using definitions  from \cite{rogers2011higher}; see also \cite[Lemma 4.3.3]{baez2004higher}.

\begin{definition}  \label{d:lietwo}
A \emph{Lie 2-algebra} is a 2-term chain complex $\VV=[V_1 \stackrel{d}{\longrightarrow}V_0]$ equipped with a skew-symmetric chain map $[-,-]: \VV \otimes \VV \to \VV$ and a skew symmetric chain homotopy $J:\VV^{\otimes 3} \to \VV$, called the \emph{Jacobiator}, from the chain map 
\[
\VV^{\otimes 3} \to \VV, \quad x\otimes y \otimes z \mapsto [x,[y,z]],
\]
to the chain map 
\[
\VV^{\otimes 3} \to \VV, \quad x\otimes y \otimes z \mapsto [[x,y],z] + [y,[x,z]],
\]
satisfying 
\begin{align*}
[x, & J(y,z,w)]  + J(x,[y,z],w) + J(x,z,[y,w]) + [J(x,y,z),w] 
+ [z,J(x,y,w)] \\
&= J(x,y,[z,w]) + J([x,y],z,w) + [y,J(x,z,w)] +J(y,[x,z],w) + J(y,z,[x,w]).
\end{align*}
\end{definition}

A Lie 2-algebra with vanishing Jacobiator is called a strict Lie-2-algebra.

\begin{definition} \label{d:lietwomap}
A \emph{morphism of Lie 2-algebras} $\mathsf{F}:\VV \to \WW$ consists of a chain map $\mathsf{F}_\bullet:V_\bullet \to W_\bullet$ together with a chain homotopy $F:\VV\otimes \VV \to \WW$, from the chain map
\[
\VV \otimes \VV \to \WW \quad x\otimes y \mapsto \mathsf{F}([x,y])
\]
to the chain map
\[
\VV \otimes \VV \to \WW \quad x\otimes y \mapsto [\mathsf{F}(x),\mathsf{F}(y)],
\]
satisfying
\begin{align} \label{eq:morph_Jac_ch}
\mathsf{F}_1(J(x,y,z)) - J(\mathsf{F}_0(x),&\mathsf{F}_0(y),\mathsf{F}_0(z)) = F(x,[y,z]) - F([x,y],z) - F(y,[x,z]) \nonumber \\
&-[F(x,y),\mathsf{F}_0(z)] + [\mathsf{F}_0(x),F(y,z)] - [\mathsf{F}_0(y),F(x,z)].
\end{align}
\end{definition}

For completeness,  we note that given a pair of morphisms of Lie 2-algebras, $\mathsf{F}, \mathsf{G}:\VV \to \WW$, a \emph{2-morphism} between $\mathsf{F}$ and $\mathsf{G}$  is a chain homotopy between them that is compatible with the underlying chain homotopies $F$ and $G$.   The 2-category of Lie 2-algebras with the above morphisms and 2-morphisms is denoted by $\LieBC$ \cite{baez2004higher}.  

We will work with the bicategory $\LieButterfly$ defined by Noohi \cite{noohi2013integrating}, which we recall next.  The objects remain the same; however, the arrows are replaced with butterflies.

\begin{definition} \label{d:butterfly} 
A \emph{butterfly} $\mathsf{E}:\VV \dashto \WW$ of Lie 2-algebras is a vector space $E$ equipped with a skew-symmetric bracket $[-,-]$, together with a commutative diagram
\[
\xymatrix@R=1em{
V_1 \ar[dd] \ar[dr]^-{\kappa} 	& 	& W_1 \ar[dd] \ar[dl]_-\lambda \\
	& E\ar[dl]^-\sigma \ar[dr]_-\rho	&	\\
V_0	&	& W_0
}
\]
such that:
\begin{itemize}
\item
$\rho\circ \kappa =0$ and the NE--SW sequence
\(
0 \to W_1 \to E \to V_0 \to 0
\)
is short exact;
\item
 $\rho$ and $\sigma$ preserve brackets;
 \item 
 for every $a\in E$, $u\in W_1$, $x\in V_1$, 
 \[
 [a,\lambda(u)] = \lambda([\rho(a),u]), \quad \text{and} \quad [a, \kappa(x)] = \kappa([\sigma(a),x]);
 \]
 \item
 for every $a,b,c\in E$,
 \[
 \lambda J(\rho(a),\rho(b),\rho(c)) + \kappa J(\sigma(a),\sigma(b),\sigma(c)) = [a,[b,c]] + [b,[c,a]] + [c,[a,b]].
 \]
\end{itemize}
\end{definition}

\begin{definition} \label{d:lietwo2arrow}
Given a pair of butterflies $\mathsf{E},\mathsf{F}:\VV\dashto \WW$, a \emph{morphism of butterflies} $\varphi: \mathsf{E} \Rightarrow \mathsf{F}$ is a linear map $E \to F$ that commutes with the brackets and all the structure maps of the butterflies. (As noted in \cite{noohi2013integrating}, the linear map in a morphism of butterflies is necessarily an isomorphism.) 
\end{definition}

The composition of butterflies
\(
\xymatrix{
\UU \ar@{-->}[r]^{\mathsf{E}} &\VV \ar@{-->}[r]^{\mathsf{E}'} & \WW
}
\)
is given  by
\[
\xymatrix@R=1em{
U_1 \ar[dd] \ar[dr]^-{(\kappa,0)} 	& 	& W_1 \ar[dd] \ar[dl]_-{(0,\lambda')} \\
	& {E \mathrel{\substack{{V_1}\\\oplus\\{V_0}}} E'} \ar[dl]^-{\sigma \circ \pr} \ar[dr]_-{\rho'\circ \pr}	&	\\
U_0	&	& W_0
}
\]
where $E \mathrel{\substack{{V_1}\\\oplus\\{V_0}}} E'$ denotes the quotient of $E\oplus_{V_0} E'$ by the image of $V_1$ under  $(\lambda,\kappa')$, with bracket defined on components.

As shown in \cite{noohi2013integrating}, Lie 2-algebras with butterflies as 1-arrows and morphism of butterflies as 2-arrows form a bicategory $\LieButterfly$. Moreover, Noohi explicitly defines a functor $\LieBC \to \LieButterfly$, and shows that this descends to an equivalence of categories on the localization of $\LieBC$ with respect to quasi-isomorphisms (i.e., morphisms inducing an isomorphism on homology). In particular, quasi-isomorphisms correspond to invertible butterflies---butterflies whose NW-SE sequence is also short exact.  (Recall that the inverse of an invertible butterfly is obtained by simply `flipping' the butterfly diagram along the vertical axis.)

We will implicitly make use of this functor in what follows---for instance, when writing down 1-arrows of Lie 2-algebras, we may write a morphism of Lie 2-algebras, rather than the corresponding butterfly.  For clarity, however, we will distinguish between the two kinds of 1-arrows by employing either a solid or dashed arrow for morphisms or butterflies, respectively.

\subsection{Lie 2-algebras associated to pre-2-plectic manifolds and bundle gerbes} \label{ss:egsLie2alg}

Lie 2-algebras appear naturally in 2-plectic geometry \cite{rogers20132plectic}, and as infinitesimal symmetries of bundle gerbes \cite{collier2011infinitesimal}. 
This section establishes Theorem \ref{p:3curv_obs} realizing the 3-curvature of a bundle gerbe as the obstruction for the horizontal lift determined by a connective structure to be a Lie 2-algebra morphism.  In Proposition \ref{p:multivfLie2}, we provide a Lie 2-algebra structure on the connection-preserving multiplicative vector fields on a bundle gerbe. For convenience, we also recall the Poisson-Lie 2-algebra associated to a pre-2-plectic manifold, which will be used in Section \ref{s:kost}.

\subsubsection*{Lie 2-algebras of multiplicative vector fields on bundle gerbes} 

As mentioned in Section \ref{ss:mult}, the category $\mathbb{X}(\mathbf{G})$ of multiplicative vector fields on a Lie groupoid $\mathbf{G}=G_1 \toto G_0$  is a (strict) Lie 2-algebra  \cite{berwick2016lie}.  Indeed, viewing $\mathbb{X}(\mathbf{G})$ as a 2-term chain complex 
\[
\Gamma(A) \longrightarrow \mathbb{X}(\mathbf{G})_0,
\]
the bracket of elements in degree 0 is given on components: 
$$
[({\mathbf{x}_0}, {\mathbf{x}_1}), ({\mathbf{x}_0'}, {\mathbf{x}_1'})] = ([{\mathbf{x}_0},{\mathbf{x}_0'}], [{\mathbf{x}_1},{\mathbf{x}_1'}]),
$$
while on $\mathbb{X}(\mathbf{G})_0\otimes \Gamma(A)$, it is given by 
$$
[({\mathbf{x}_0}, {\mathbf{x}_1}),a] = [{\mathbf{x}_1},\overrightarrow{a}]\big|_{G_0}.
$$
 (The bracket on $\Gamma(A)\otimes \mathbb{X}(\mathbf{G})_0$ is determined by skew-symmetry.) Here, recall that for   $a\in \Gamma(A)$, and $({\mathbf{x}_0},{\mathbf{x}_1})$ multiplicative, $[{\mathbf{x}_1},\overrightarrow{a}] \in \ker ds$ and is right-invariant \cite{mackenzie1998classical}, and hence its restriction to $G_0$ defines a section in $\Gamma(A)$.

Therefore, associated to a bundle gerbe $\calG=(X,P,\mu)$ over a manifold $M$, we naturally obtain (strict) Lie 2-algebras $\mathbb{X}(X^{[2]} \toto X)$ and $\mathbb{X}(P\toto X)$ as above. Recall from Proposition \ref{p:curv_measures_bracket_preserve_failure} (1) that a connective structure $(B,\gamma)$ on $\calG$ give rise to a map of multiplicative vector fields $\mathrm{Lift}_\gamma:\mathbb{X}(X^{[2]} \toto X)_0 \to \mathbb{X}(\mathbf{P})_0$, which (\emph{cf}.\ the proof of Proposition \ref{p:properties} (2)) may be completed to a chain map 
\begin{equation} \label{e:chmap}
\mathrm{Lift}_{\gamma,\bullet}:\mathbb{X}(X^{[2]} \toto X)_\bullet \longrightarrow \mathbb{X}(\mathbf{P})_\bullet.
\end{equation}
By Proposition \ref{p:curv_measures_bracket_preserve_failure} (2), the curving $B$ gives rise to a map
\begin{equation} \label{e:chhomotopy}
F_B:\mathbb{X}(X^{[2]} \toto X)_0\times \mathbb{X}(X^{[2]} \toto X)_0 \longrightarrow \Gamma(A_P)
\end{equation}
by setting $F_B({\mathbf{x}},{\mathbf{x}}^{[2]};{\mathbf{x}'},{\mathbf{x}'}^{[2]})=\iota_{{\mathbf{x}}}\iota_{{\mathbf{x}'}}B \genz\big|_X$, which we may view as defining a chain homotopy that controls the failure of $\mathrm{Lift}_\gamma$ to preserve the bracket (as in Definition \ref{d:lietwomap}).
 However, $\mathrm{Lift}_\gamma$ together with $F_B$ fails to be a morphism of Lie 2-algebras (as in Definition \ref{d:lietwomap}). Indeed,  analogous to the horizontal lift for principal $S^1$-bundles, the failure of the horizontal lift to be a morphism of Lie 2-algebras (i.e.,  Condition \eqref{eq:morph_Jac_ch} in Definition \ref{d:lietwomap}) is measured by the 3-curvature of the bundle gerbe, as the next Theorem shows.
 
\begin{theorem} \label{p:3curv_obs}
Let $(\calG,B,\gamma)$ be a bundle gerbe over $M$ with connective structure and 3-curvature $\chi$. The horizontal lift  $\mathrm{Lift}_\gamma$ and chain homotopy $F_B$ from \eqref{e:chmap} and \eqref{e:chhomotopy} define a Lie 2-algebra morphism if and only if $\chi=0$. 
\end{theorem}
 \begin{proof}
Following the paragraph preceding the Proposition, it remains to consider Condition \eqref{eq:morph_Jac_ch} from Definition \ref{d:lietwomap}. To that end, consider vector fields $\xi_1$, $\xi_2$, and $\xi_3$ on $M$ with corresponding lifts  ${\mathbf{x}}_1$, ${\mathbf{x}}_2$, and ${\mathbf{x}}_3$ on $X$. Write ${\overline{\mathbf{x}}} = ({\mathbf{x}},{\mathbf{x}}^{[2]})$. 

Since the Lie 2-algebras under consideration are strict (i.e., with vanishing Jacobiator), $\mathrm{Lift}_\gamma$ and $F_B$ will define a Lie 2-algebra morphism if and only if
\begin{align}
0=&
F_B({\overline{\mathbf{x}}}_1,[{\overline{\mathbf{x}}}_2,{\overline{\mathbf{x}}}_3]) - F_B([{\overline{\mathbf{x}}}_1,{\overline{\mathbf{x}}}_2],{\overline{\mathbf{x}}}_3) - F_B({\overline{\mathbf{x}}}_2,[{\overline{\mathbf{x}}}_1,{\overline{\mathbf{x}}}_3]) \nonumber \\
& - [F_B({\overline{\mathbf{x}}}_1,{\overline{\mathbf{x}}}_2), \mathrm{Lift}_\gamma({\overline{\mathbf{x}}}_3)] + [ \mathrm{Lift}_\gamma({\overline{\mathbf{x}}}_1) , F_B({\overline{\mathbf{x}}}_2,{\overline{\mathbf{x}}}_3)] -  [ \mathrm{Lift}_\gamma({\overline{\mathbf{x}}}_2) , F_B({\overline{\mathbf{x}}}_1,{\overline{\mathbf{x}}}_3)]. \label{eq:dB1}
\end{align}
Since the image of $F_B$ consists of vertical sections of $A_P$, each of the last three terms above is also a vertical section. To determine the vertical components of these terms, let $a_{jk}=F_B({\overline{\mathbf{x}}}_j,{\overline{\mathbf{x}}}_k)$, and ${\mathbf{p}}_i = \mathrm{Lift}_\gamma({\overline{\mathbf{x}}}_i)$, and observe that
\(
\iota_{[{\mathbf{p}}_i,\overrightarrow{a}_{jk}]} \gamma = L_{{\mathbf{p}}_i} \iota_{\overrightarrow{a}_{jk}} \gamma ,
\)
since ${\mathbf{p}}_i$ is horizontal and $\overrightarrow{a}_{jk}$ is vertical. Restricting to $X$, we see that 
\[
\iota_{[{\mathbf{p}}_i,\overrightarrow{a}_{jk}]} \gamma \big|_X = 
L_{{\mathbf{x}}_i} \iota_{{\mathbf{x}}_j} \iota_{{\mathbf{x}}_k} B
\]
Therefore, taking vertical components in \eqref{eq:dB1}, we obtain
\begin{align*}
0&=
\iota_{{\mathbf{x}}_1} \iota_{[{\mathbf{x}}_2,{\mathbf{x}}_3]} B - \iota_{[{\mathbf{x}}_1,{\mathbf{x}}_2]} \iota_{{\mathbf{x}}_3} B - \iota_{{\mathbf{x}}_2} \iota_{[{\mathbf{x}}_1,{\mathbf{x}}_3]} B + L_{{\mathbf{x}}_3} \iota_{{\mathbf{x}}_1} \iota_{{\mathbf{x}}_2} B +
L_{{\mathbf{x}}_1} \iota_{{\mathbf{x}}_2} \iota_{{\mathbf{x}}_3} B -
L_{{\mathbf{x}}_2} \iota_{{\mathbf{x}}_1} \iota_{{\mathbf{x}}_3} B \\
&=\iota_{{\mathbf{x}}_1} \iota_{{\mathbf{x}}_2} \iota_{{\mathbf{x}}_3} dB
\end{align*}
Hence Condition \eqref{eq:morph_Jac_ch} holds if and only if $\pi^*\chi=dB=0$, which completes the proof, since $\pi:X\to M$ is a submersion.
\end{proof}

\begin{remark} \label{r:3curvCollier}
The interpretation of the 3-curvature as in Theorem \ref{p:3curv_obs} essentially appears in Collier \cite[Section 3.6]{collier-PhDThesis2012} from the stacky perspective, by appealing to the \v{C}ech data implementation of the  category of \emph{infinitesimal symmetries} of $S^1$-gerbes (\emph{cf}.\ Remark \ref{r:Lie2Cechdata_multvf} below). 
\eoe
\end{remark}

The category $\mathbb{X}(\mathbf{P};B,\gamma)$ from Corollary \ref{c:subcat} of connection-preserving multiplicative vector fields is  a 2-vector space.  Proposition \ref{p:multivfLie2} below shows it is a Lie 2-algebra, whose brackets are inherited from $\mathbb{X}(\mathbf{P})$.

\begin{lemma} \label{l:vbracket}
Let $\calG=(X,P,\mu)$ be a bundle gerbe over $M$ with connective structure $(B,\gamma)$, and $P\toto X$ the associated Lie groupoid.
Let $({\mathbf{x}},{\mathbf{p}};\alpha) \in \mathbb{X}(\mathbf{P};B,\gamma)_0$  and $a \in \Gamma(A_P)$. Then $ \mathsf{v}_{[{\mathbf{p}},\overrightarrow{a}]\big|_X}  =L_{{\mathbf{x}}} \mathsf{v}_a +  \iota_{{\mathbf{a}}}  \alpha$.
\end{lemma}
\begin{proof} 
First, recall that $\overrightarrow{a} \sim_{\delta_1} {\mathbf{a}}$ and $\overrightarrow{a} \sim_{\delta_0} 0$. Therefore,
\[
\mathsf{v}_{[{\mathbf{p}},\overrightarrow{a}]\big|_X} 
= \epsilon^* \iota_{[{\mathbf{p}},\overrightarrow{a}]} \gamma 
= \epsilon^*(L_{{\mathbf{p}}} \iota_{\overrightarrow{a}} \gamma -\iota_{\overrightarrow{a}} L_{{\mathbf{p}}}  \gamma) 
= L_{{\mathbf{x}}} \mathsf{v}_a - \epsilon^* \iota_{\overrightarrow{a}} \delta \alpha 
= L_{{\mathbf{x}}} \mathsf{v}_a +  \iota_{{\mathbf{a}}}  \alpha .
\]
 \end{proof}

\begin{proposition} \label{p:multivfLie2}
Let $\calG=(X,P,\mu)$ be a bundle gerbe over $M$ with connective structure $(B,\gamma)$.  Let $\mathbb{X}(\calG; B,\gamma) = \{V_1 \to V_0 \}$   denote the  2-term complex
\[
\Gamma(A_P) \stackrel{d}{\longrightarrow} \mathbb{X}(\mathbf{P};B,\gamma)_0,
\]
with differential given by
 $da=({\mathbf{a}},{\bar{\mathbf{a}}}; \iota_{{\mathbf{a}}} B - d\mathsf{v}_a)$. 
 Define a bracket on $V_0\otimes V_0$  by
\[
[({\mathbf{x}}, {\mathbf{p}}; \alpha), ({\mathbf{z}}, {\mathbf{r}}; \beta)] = ([{\mathbf{x}},{\mathbf{z}}], [{\mathbf{p}},{\mathbf{r}}]; L_{{\mathbf{x}}}\beta - L_{{\mathbf{z}}}\alpha )
\]
and  on $V_0\otimes V_1$  by
\[
[({\mathbf{x}}, {\mathbf{p}};\alpha),a] = [{\mathbf{p}},\overrightarrow{a}]\big|_{X}.
\]
Then  $\mathbb{X}(\calG; B,\gamma)$ is a strict Lie 2-algebra (i.e., with Jacobiator $J\equiv 0$). 
\end{proposition}
\begin{proof}
Let $\VV=\{ V_1 \to V_0\}$ as in the statement of the Proposition. 
A direct calculation verifies that the bracket and differential are well-defined. To show the resulting bracket is compatible with the differential (i.e., that it gives a chain map $\mathbb{V}\otimes \mathbb{V} \to \mathbb{V}$), we must verify that for $({\mathbf{x}}, {\mathbf{p}};\alpha) \in V_0$ and $a,b\in V_1$:
\begin{align}
d\,  [({\mathbf{x}}, {\mathbf{p}};\alpha),a] &= [({\mathbf{x}}, {\mathbf{p}};\alpha),da], \quad \text{and} \label{eq:d1} \\
[d a, b] & =- [d b, a]. \label{eq:d2}
\end{align}
For \eqref{eq:d1}, first recall that $c=([{\mathbf{p}},\overrightarrow{a}]\big|_X)$ defines a section of $A_P$ and that $\mathbf{c}= [{\mathbf{x}},{\mathbf{a}}]$ and that $\bar{\mathbf{c}} = [{\mathbf{p}},{\bar{\mathbf{a}}}]$ (e.g., see the proof of \cite[Theorem 3.4]{berwick2016lie}).  Therefore, using Lemma \ref{l:vbracket},
\begin{align*}
d\,  [({\mathbf{x}}, {\mathbf{p}};\alpha),a] &= ([{\mathbf{x}},{\mathbf{a}}] , [{\mathbf{p}},{\bar{\mathbf{a}}}],\iota_{[{\mathbf{x}},{\mathbf{a}}]} B -d\mathsf{v}_{[{\mathbf{p}},\overrightarrow{a}]\big|_X} ) \\
&=([{\mathbf{x}},{\mathbf{a}}] , [{\mathbf{p}},{\bar{\mathbf{a}}}], 
L_{{\mathbf{x}}} \iota_{{\mathbf{a}}} B -   \iota_{{\mathbf{a}}} L_{{\mathbf{x}}} B - d(L_{{\mathbf{x}}} \mathsf{v}_a +  \iota_{{\mathbf{a}}}  \alpha)) \\
&= ([{\mathbf{x}},{\mathbf{a}}] , [{\mathbf{p}},{\bar{\mathbf{a}}}], 
L_{{\mathbf{x}}}( \iota_{{\mathbf{a}}} B  -  d\mathsf{v}_a) -  L_{{\mathbf{a}}} \alpha) \\
&= [({\mathbf{x}}, {\mathbf{p}};\alpha),da].
\end{align*}

For \eqref{eq:d2}, we first recall that $[{\bar{\mathbf{a}}},\overrightarrow{b}] = [\overrightarrow{a} + \overleftarrow{a},\overrightarrow{b}] = [\overrightarrow{a} ,\overrightarrow{b}]$, since left and right invariant vector fields commute. Therefore,
\[
[da,b] = [{\bar{\mathbf{a}}},\overrightarrow{b}]\big|_X 
=-[\check{b},\overrightarrow{a}]\big|_X 
=-[db,a], 
\]
as required.
The vanishing of the Jacobiator is straightforward.
\end{proof}

 \begin{example} \label{eg:Lie2-trivgerbe}
 Consider the connection-preserving multiplicative vector fields for the trivial bundle gerbe $\calI_\omega$ from Example \ref{eg:multivf_trivgerbe}.  As a Lie 2-algebra $\mathbb{X}(\calI_\omega) =\{W_1\to W_0\}$ has $W_1\cong C^\infty(M)$, since  vectors in the Lie algebroid must necessarily be vertical vectors,  and $ W_0$ consists of pairs $(\xi,A)$ such that $L_\xi \omega=dA$.  The differential $d:W_1\to W_0$ is $f\mapsto (0, -df)$.
 The bracket on $W_0\otimes W_0$ is given by
 \[
 [(\xi_1,A_1),(\xi_2,A_2)] = ([\xi_1,\xi_2],L_{\xi_1}A_2 - L_{\xi_2} A_1),
 \]
 and on $W_0\otimes W_1$ it is given by
 \[
[(\xi,A),f] = \xi(f). 
 \]

Note that there is a morphism of Lie 2-algebras $\mathsf{F}:\mathbb{X}(\calI_\omega) \to \mathbb{L}(M,d\omega)$, which in degree 0 sends $(\xi,A) \mapsto (\xi,\iota_\xi \omega +A)$. A routine calculation shows that $\mathsf{F}$ is a quasi-isomorphism. Theorem \ref{t:invertbutterflymulti-rogers} generalizes this observation.
\eoe
\end{example}

\begin{remark} \label{r:Lie2Cechdata_multvf}
In Example \ref{eg:collier:morphisms}, we saw that the category of multiplicative vector fields $\mathbb{X}(\mathbf{P})$ on a bundle gerbe $P$ given by a \v{C}ech 2-cocycle $g_{ijk}$ recovers the category $\mathcal{L}_{g_{ijk}}$ of infinitesimal symmetries from \cite{collier2011infinitesimal}.  In \emph{op.\ cit.\ }, Collier also shows $\mathcal{L}_{g_{ijk}}$ is a Lie 2-algebra---the Lie 2-algebra structure coincides with $\mathbb{X}(\mathbf{P})$ from \cite{berwick2016lie}.
The analogous statement holds as well for the  category  $\mathcal{L}_{(g_{ijk},B_i,A_{ij})}$ of infinitesimal symmetries preserving the curving and $\mathbb{X}(\calG_{g_{ijk}},\{B_i\},\{A_{ij}\})$.
\eoe
\end{remark}

\subsubsection*{The Poisson-Lie 2-algebra} 

Recall that a \emph{2-plectic manifold} is a manifold equipped with a closed non-degenerate 3-form (in analogy with symplectic manifolds and closed non-degenerate 2-forms).  If we omit the non-degeneracy condition, we say the manifold is \emph{pre}-2-plectic \cite{rogers20132plectic}.

\begin{definition} \label{d:rogers}
Let $(M,\chi)$ be a pre-2-plectic manifold. The \emph{Poisson-Lie} 2-algebra $\LL(M,\chi)$ has underlying 2-term complex  given by
\[
\xymatrix{
 C^{\infty}(M) \ar[r]^-{(0,d)} &
\{ (\xi,\beta) \in \mathfrak{X}(M) \times \Omega^1(M) \, | \, \iota_\xi \chi = -d\beta\}
}.
\]
with bracket given by 
\[
[(\xi_1,\beta_1),(\xi_2,\beta_2)] = ([\xi_1,\xi_2],\iota_{\xi_2}\,\iota_{\xi_1} \chi)
\]
on $L_0\otimes L_0$ and zero otherwise. The Jacobiator $J:L_0^{\otimes 3} \to L_1$ is given by 
\[
J(\xi_1,\beta_1;\xi_2,\beta_2;\xi_3,\beta_3) = -\iota_{\xi_3}\iota_{\xi_2}\iota_{\xi_1}\chi.
\]
\end{definition}

\section{Applications to 2-plectic and quasi-Hamiltonian geometry} \label{s:app}

In this section, we establish some 2-plectic and quasi-Hamiltonian analogues of  classical results in symplectic geometry. In Theorem \ref{t:invertbutterflymulti-rogers}, we show the Poisson Lie 2-algebra of a 2-plectic manifold $(M,\chi)$ is canonically quasi-isomorphic to the connection-preserving multiplicative vector fields on a bundle gerbe with connective structure and 3-curvature $\chi$.  Using Theorem \ref{t:invertbutterflymulti-rogers}, we obtain Corollary \ref{c:2KostantLift} and Theorem \ref{t:qHamKost}, analogues of a formula of Kostant lifting Lie algebra actions to infinitesimal symmetries of prequantizations in the contexts of 2-plectic and quasi-Hamiltonian geometry.

\subsection{Infinitesimal symmetries of $S^1$-gerbes and  pre-2-plectic manifolds} \label{ss:quasi-isos}
 
 In this section, we establish Theorem \ref{t:invertbutterflymulti-rogers}, which shows that the Poisson-Lie 2-algebra of a pre-2-plectic manifold $(M,\chi)$  is canonically quasi-isomorphic to the Lie 2-algebra of connection preserving multiplicative vector fields on a bundle gerbe with connective structure and 3-curvature $\chi$.
 
Theorem \ref{t:invertbutterflymulti-rogers}  is the 2-plectic analogue of the following well-known result due to Kostant.  Let $(M,\omega)$ be a symplectic manifold,  and suppose we are given a prequantization: a principal $S^1$-bundle $\pi: P\to M$ with connection $\gamma \in \Omega^1(P)$ whose curvature $\mathrm{curv}(\gamma)=\omega$.  Then the formula
\[
f\mapsto \mathrm{Lift}_\gamma(X_f) + \pi^*f\,\genz
\]
defines a Lie algebra isomorphism  $\xymatrix{C^\infty(M) \ar[r]^-{\cong} & \mathcal{Q}(P,\gamma)}$, where $\mathrm{Lift}_\gamma(X_f)$ denotes the horizontal lift of $X_f$, the Hamiltonian vector field of $f$, and  
$$\mathcal{Q}(P,\gamma) = \{ X\in \mathfrak{X}(P)\, : \, L_X\gamma =0 \}$$ denotes the so-called \emph{infinitesimal quantomorphisms} of $P$ \cite{kostant1970quantization,vaughan2015metaplectic}. 

Before arriving at the 2-plectic analogue, Theorem \ref{t:invertbutterflymulti-rogers} below, we first note that given a multiplicative vector field $({\mathbf{x}},{\mathbf{p}})$ on a bundle gerbe $\calG=(X,P,\mu)$ with connective structure $(B,\gamma)$, 
$\delta(\iota_{{\mathbf{p}}} \gamma)= \iota_{({\mathbf{p}},{\mathbf{p}})} \delta \gamma =0$. Therefore, since $P\toto X$ is proper, there exists $g:X\to \RR$ such that $\delta g = \iota_{{\mathbf{p}}} \gamma$. 

Let 
\(
E=\{
({\mathbf{x}},{\mathbf{p}};\alpha,g) \in V_0 \times C^\infty(X) \, \big|\, \delta g = \iota_{{\mathbf{p}}} \gamma
\}
\)
and define structure maps in the diagram,
\[
\xymatrix@R=1em{
L_1 \ar[dd] \ar[dr]^-{\kappa} 	& 	& V_1 \ar[dd] \ar[dl]_-{\lambda}\\
	& E\ar[dl]^-\sigma \ar[dr]_-\rho	&	\\
L_0	&	& V_0
}
\]
 as follows. Let $\rho=\pr_1$ be the obvious projection, $\lambda(a) = (da,-\mathsf{v}_a)$, and $\kappa(f) = (0,0;0,\pi^*f)$. To define $\sigma$, first observe that $\delta(\alpha-\iota_{{\mathbf{x}}}\, B - dg) =0$, and hence there exists a unique 1-form $\varepsilon \in \Omega^1(M)$ such that $\pi^*\varepsilon = \alpha-\iota_{{\mathbf{x}}}\, B - dg$. Let $\xi$ denote the vector field with $\mathbf{x} \sim_\pi \xi$. Then $\pi^*\iota_{\xi}\chi = \iota_{{\mathbf{x}}}\, dB = \pi^*d\varepsilon$. 
Set $\sigma({\mathbf{x}},{\mathbf{p}};\alpha,g) = (\xi,-\varepsilon)$.

\begin{theorem} \label{t:invertbutterflymulti-rogers}
Let $(M,\chi)$ be a pre-2-plectic manifold, and let $\calG=(X,P,\mu)$ be a bundle gerbe with connective structure $(B,\gamma)$ and 3-curvature $\chi$. Let $E$ be as above, with bracket given by
\[
[({\mathbf{x}},{\mathbf{p}},\alpha,g),({\mathbf{z}},{\mathbf{r}},\beta,h)] = ([({\mathbf{x}},{\mathbf{p}},\alpha),({\mathbf{z}},{\mathbf{r}},\beta)], \iota_{{\mathbf{x}}} \beta - \iota_{{\mathbf{z}}} \alpha  + \iota_{{\mathbf{z}}} \iota_{{\mathbf{x}}} B).
\]
Then $E$ defines a natural invertible butterfly $\mathsf{E}: \mathbb{L}(M,\chi) \dashto \mathbb{X}(\calG; B,\gamma)$.
\end{theorem}
\begin{proof} 
To check that $E$ defines a butterfly, we verify the required compatibility of the bracket with the structure maps. Clearly $\sigma$ preserves the brackets. Also, 
$$
[({\mathbf{x}},{\mathbf{p}},\alpha,g), \kappa(h)] = [({\mathbf{x}},{\mathbf{p}},\alpha,g), (\overline{0},0,\pi^*h)] = 0 = \kappa[(\xi,\varepsilon),h],
$$ 
since brackets of elements of mixed degree vanish in $\LL(M,\chi)$.  By Lemma \ref{l:vbracket} and the properties of the bracket for $\mathbb{X}(\calG, B,\gamma)$, we have $[({\mathbf{x}},{\mathbf{p}},\alpha,g), \lambda(a)] = \lambda[({\mathbf{x}},{\mathbf{p}},\alpha),a]$. Finally, to see that $\sigma$ preserves the brackets, 
it suffices to verify that 
\[
-\pi^* \iota_\zeta \iota_\xi \chi =  L_{{\mathbf{x}}} \beta - L_{{\mathbf{z}}} \alpha - \iota_{[{\mathbf{x}},{\mathbf{z}}]} B - d(\iota_{{\mathbf{x}}} \beta - \iota_{{\mathbf{z}}} \alpha  + \iota_{{\mathbf{z}}} \iota_{{\mathbf{x}}} B),
\]
where $\zeta$ denotes the vector field on $M$ with $\mathbf{z}\sim_\pi \zeta$,
which follows  from standard applications of Cartan calculus on differential forms.

Since the Jacobiator for $\mathbb{X}(\calG,B,\gamma)$ is trivial, the Jacobiator condition on $\mathsf{E}$ simplifies to
\[
\kappa J_{\mathbb{L}}(\xi_1,\varepsilon_1;\xi_2,\varepsilon_2;\xi_3,\varepsilon_3) = [({\mathbf{x}}_1,{\mathbf{p}}_1;\alpha_1), [ ({\mathbf{x}}_2,{\mathbf{p}}_2;\alpha_2),({\mathbf{x}}_3,{\mathbf{p}}_3;\alpha_3)]] + \textit{cyc.\ perm.\ }
\]
The $V_0$-component in the above equality follows from the vanishing of the Jacobiator for $\mathbb{X}(\calG, B, \gamma)$. To verify the component in $C^\infty(X)$, it should be verified that
\[
-\pi^*\iota_{\xi_3}\iota_{\xi_2}\iota_{\xi_1}\chi = \iota_{{\mathbf{x}}_1} (L_{{\mathbf{x}}_2} \alpha_3 - L_{{\mathbf{x}}_3} \alpha_2) - \iota_{[{\mathbf{x}}_2,{\mathbf{x}}_3]} \alpha_1 + \iota_{[{\mathbf{x}}_2,{\mathbf{x}}_3]} \iota_{{\mathbf{x}}_1} B + \textit{cyc.\ perm.,\ }
\]
which is a straightforward, albeit tedious, application of Cartan calculus on differential forms---see Proposition \ref{p:jacobiator}.

To show $\mathsf{E}$ is an invertible butterfly, we must check that the diagonal sequences are short exact. That the NW--SE sequence is short exact follows easily from properness of the Lie groupoid $P\toto X$.  Exactness of  the NE--SW sequence is verified next.

To check surjectivity of $\sigma$: given  $(\xi,\varepsilon)\in L_0$, choose a multiplicative lift of $\xi$ and take $\alpha' = -\pi^*\varepsilon + \iota_{{\mathbf{x}}} B$. Next, find a function $f$ with $\delta f = \iota_{{\mathbf{p}}} \gamma$ and  set  $\alpha=\alpha' +df$.  Then $\sigma({\mathbf{x}},{\mathbf{p}};\alpha,f)=(\xi,\varepsilon)$

To check injectivity of $\lambda$: suppose $a\in \Gamma (A_P)$ is a section of the Lie algebroid and $da=0$ and $\mathsf{v}_a =0$.  From $da=0$, we get that ${\mathbf{a}} = 0$, which gives that $a$ lies in the the isotropy Lie algebra (in this case trivial, i.e., $X\times \RR$), whence we view $a$ as a function on $X$, and we get that ${\bar{\mathbf{a}}}=0$, which gives that $\delta a=0$ and hence $a$ descends to $M$. From $\mathsf{v}_a=0$, we see that the vertical part of $a$ must vanish, but since ${\mathbf{a}}=0$ all of $a$ is vertical, so $a$ must vanish.

Exactness at $E$: It's easy to see that $\sigma\circ \lambda=0$ since $\pi$ is a submersion. Suppose $\sigma ({\mathbf{x}},{\mathbf{p}},\alpha,g) =0$. Then $\xi=0$ and $\alpha - \iota_{{\mathbf{x}}} B - dg=0$. The first of those gives  $({\mathbf{x}},{\mathbf{p}})=({\mathbf{a}},{\bar{\mathbf{a}}})$ for some $a\in \Gamma(A_P)$ (See Proposition \ref{p:properties}.)  Since $\delta \mathsf{v}_a = \delta g$, there exists a (unique) $f:M \to \RR$ such that $\pi^*f = g-\mathsf{v}_a$. We may view $f$ as defining a section of the isotropy Lie algebra, as noted in the previous paragraph, by setting $b=\pi^*f$. That is, $b:X\to A_P$ is a vertical vector in $TP\big|_X$ given by $b(x) = f(\pi(x)) \genz \big|_{\epsilon(x)}$.  Note that $\mathsf{v}_b = \pi^*f$, and  $\widetilde{b}=0$. Moreover, $\check{b}=0$ by Proposition \ref{p:properties} . Therefore, $\lambda(a+b)=({\mathbf{x}},{\mathbf{p}},\alpha,g)$.
\end{proof}

\begin{remark} 
Given an element $({\mathbf{x}},{\mathbf{p}};\alpha,g)$ in $E$, 
by Proposition \ref{p:properties}, the function $g:X\to \RR$ gives rise to an isomorphism of connection-preserving multiplicative vector fields $a:(\mathbf{x},\mathbf{p},\alpha) \to (\mathbf{x}, \mathrm{Lift}_\gamma(\mathbf{x}^{[2]}), \alpha-dg)$, where $\mathrm{Lift}_\gamma (\mathbf{x}^{[2]})$ denotes the horizontal lift determined by the connection $\gamma$. (\emph{Cf}.\ \cite[Lemma 2.4]{syvestre2020prequantisation}.) Thus we may view the butterfly $E$ as encoding all possible isomorphisms between  connection-preserving multiplicative vector fields and their horizontal projections, which in turn canonically descends to  Hamiltonian vector fields on the underlying manifold $M$.
\eoe \end{remark}

\begin{remark}\label{r:Ehresmann}
A choice of Ehresmann connection $H\subset TX$ on the submersion $\pi:X\to M$ of the bundle gerbe $\calG$ gives rise to a section $s:L_0 \to E$ of the map $\sigma$ of the butterfly $E$ of Theorem \ref{t:invertbutterflymulti-rogers}, by setting
\[
s(\xi,\beta) = (\mathrm{Hor}(\xi), \mathrm{Lift}_\gamma(\mathrm{Hor}(\xi)^{[2]}); \iota_{\mathrm{Hor}(\xi)} B -\pi^*\beta, 0),
\]
where $\mathrm{Hor}(\xi)$ denotes the horizontal lift of $\xi$ with respect to the connection $H$.
Applying \cite[Proposition 3.4]{noohi2013integrating}, one may write the above quasi-isomorphism as a morphism of Lie 2-algebras. Moreover, the 2-isomorphism class of such a morphism is independent of the choice of connection, since the butterfly $E$ is as well.
\eoe
\end{remark}

\begin{remark} \label{r:FRS}
In \cite[Thm 4.6]{fiorenza2014algebras}, the authors prove, in the more general context of  $n$-plectic geometry and $L_\infty$-algebras,
a  result similar to Theorem \ref{t:invertbutterflymulti-rogers}.  Specializing to $n=2$, the result in  \emph{op.\ cit.}\ shows that $\LL(M,\chi)$ is quasi-isomorphic to a Lie 2-algebra the authors call the \emph{infinitesimal quantomorphisms} of a principal $U(1)$-2-bundle with connection. In that context, a principal $U(1)$-2-bundle is represented by a \v{C}ech-Deligne cocycle, which is equivalent to the data of an $S^1$-gerbe with connective structure given in terms of \v{C}ech data, while the corresponding Lie 2-algebra of infinitesimal quantomorphisms is essentially the Lie 2-algebra of infinitesimal symmetries preserving the connective structure from \cite{collier2011infinitesimal}. (\emph{Cf.}\ Example \ref{eg:collier:subcat} and Remark \ref{r:Lie2Cechdata_multvf}.)
\eoe \end{remark}

\begin{remark} \label{r:SW}
In  \cite{syvestre2020prequantisation}, the authors establish a quasi-isomorphism similar to Theorem \ref{t:invertbutterflymulti-rogers} for bundle gerbes  modelled as lifting bundle gerbes of $\mathsf{PU}(\calH)$-bundles over $M$, where $\calH$ is a Hilbert space and $\mathsf{PU}(\calH)$ denotes the projective unitary group. By choosing a principal connection on the given $\mathsf{PU}(\calH)$-bundle, the authors establish a  quasi-isomorphism similar to \ref{t:invertbutterflymulti-rogers} as a morphism of Lie 2-algebras. (\emph{Cf}.\ Remark \ref{r:Ehresmann}.)
\eoe 
\end{remark}

The special case of Theorem \ref{t:invertbutterflymulti-rogers} for trivial bundle gerbes $\calI_{\omega}$ can be simplified, since in that case the butterfly is induced by a morphism of Lie 2-algebras. Indeed, since the structure map $\sigma$ has a section, by \cite[Prop.\ 3.4]{noohi2013integrating} the butterfly is induced by
\[
\mathbb{L}(M,d\omega) \to \mathbb{X}(\calI_\omega), \quad (\xi,\beta) \mapsto (\xi, \iota_\xi \omega - \beta).
\]

\subsubsection*{Functorial properties of connection preserving multiplicative vector fields}

By Theorem \ref{t:invertbutterflymulti-rogers}, the Lie 2-algebra of connection-preserving multiplicative vector fields on a bundle gerbe $(\calG; B, \gamma)$ with connective structure depends only on its 3-curvature $\chi \in \Omega^3(M)$. It is straightforward to verify that the assignment $(\calG; B, \gamma) \mapsto \mathbb{X}(\calG; B, \gamma)$ may be extended to a pseudo 2-functor
\(
\mathsf{BunGerbe}^{\nabla}(M) \longrightarrow \LieBC,
\)
in the naive way, by sending an isomorphism $(Q,\gamma_Q): (\calG,B,\gamma) \dashto (\calG', B', \gamma')$ to the invertible butterfly $\mathbb{X}(Q,\gamma_Q)=\mathsf{E}'\circ \mathsf{E}^{-1}: \mathbb{X}(\calG; B,\gamma) \dashto \mathbb{X}(\calG'; B',\gamma')$, where $\mathsf{E}$ and $\mathsf{E}'$ are from Theorem \ref{t:invertbutterflymulti-rogers}, and sending all 2-isomorphisms to identities.

\subsection{Analogues of Kostant's formula lifting Lie algebra actions} \label{s:kost}
This section considers 2-plectic and quasi-Hamiltonian analogues of a formula of Kostant in symplectic geometry. These analogues, Corollary \ref{c:2KostantLift} and Theorems \ref{t:qHamKost}, are direct consequences of Theorem \ref{t:invertbutterflymulti-rogers} from the previous section.

We begin by first recalling this formula of Kostant in symplectic geometry.
Consider a Lie group $G$ acting on a symplectic manifold $(M,\omega)$, and suppose $\pi:(P,\gamma) \to (M,\omega)$ is a prequantization. Let $\g$ denote the Lie algebra of $G$.
If the $G$-action on $(M,\omega)$ is Hamiltonian, with moment map $\phi:M\to \mathfrak{g}^*$, we may compose the Kostant Lie algebra isomorphism  $C^\infty(M) \cong \mathcal{Q}(P,\gamma)$ with the (co)moment map $\langle \phi,-\rangle : \g \to C^\infty(M)$, to obtain a
lift 
of the infinitesimal action $\mathfrak{g} \to \mathfrak{X}_{\mathrm{Ham}}(M,\omega)$, $\xi\mapsto \xi_M$,  to a connection-preserving $\g$-action on $P$, according to Kostant's formula \cite{kostant1970quantization}, 
\begin{equation} \label{eq:kostant}
\mathfrak{g} \longrightarrow \mathcal{Q}(P,\gamma), \quad \xi \mapsto \mathrm{Lift}_\gamma(\xi_M) + \pi^*\langle\phi,\xi\rangle \genz.
\end{equation}

\subsubsection{Kostant's lifting formula in the 2-plectic setting} \label{ss:Kostant2plec}

Analogously, consider a Lie group $G$ acting on a pre-2-plectic manifold $(M,\chi)$.  If the $G$-action preserves $\chi$ and the corresponding $\g$-action $\xi\mapsto \xi_M$ is via Hamiltonian vector fields (i.e., there exist 1-forms $\varepsilon \in \Omega^1(M)$ such that $\iota_{\xi_M} \chi = -d\varepsilon$), the authors in \cite{callies2016homotopy} define a \emph{homotopy moment map} to be a morphism of Lie 2-algebras that lifts the Lie algebra morphism $\g \to \mathfrak{X}(M,\chi)$, $\xi \mapsto \xi_M$,
\[
\xymatrix{
	& \mathbb{L}(M,\chi) \ar[d] \\
\g \ar@{-->}[ur] \ar[r] & \mathfrak{X}(M,\chi).
}
\]
where $\mathbb{L}(M,\chi) \to \mathfrak{X}(M,\chi)$ is the natural projection, viewed as a morphism of Lie 2-algebras  \cite[Prop.\ 4.8]{callies2016homotopy}, and $\mathfrak{X}(M,\chi)$ denotes the Lie algebra of Hamiltonian vector fields on $(M,\chi)$. 

The Lie 2-algebra of connection preserving multiplicative vector fields on a bundle gerbe  also admits a natural projection $\mathbb{X}(\calG;B, \gamma) \to \mathfrak{X}(M,\chi)$, and the following diagram 2-commutes.
\begin{equation} \label{eq:proj}
\xymatrix@R=0.7em@C=0.7em{
\mathbb{L}(M,\chi) \ar@{-->}[rr]^{\mathsf{E}} \ar[ddr] & & {\mathbb{X}(\calG; B, \gamma)} \ar[ddl]\\
\rrtwocell\omit{} & & \\
	& \mathfrak{X}(M,\chi) &
}
\end{equation}
Therefore, we readily obtain the following Corollary of Theorem \ref{t:invertbutterflymulti-rogers}, which is analogous to Kostant's formula \eqref{eq:kostant} above in the context of 2-plectic geometry.

\begin{corollary} \label{c:2KostantLift}
Let $(\calG,B,\gamma)$ be a bundle gerbe over $M$ with connective structure and 3-curvature $\chi$. If a Lie group $G$ acts on $M$ preserving $\chi$ and the corresponding $\g$-action is Hamiltonian, admitting a homotopy moment map, then by composing with the invertible butterfly $\mathsf{E}$ from Theorem \ref{t:invertbutterflymulti-rogers} we obtain a
 butterfly $\g \dashto \XX(\calG;B,\gamma)$ lifting the $\g$-action on $M$---that is, the following diagram 2-commutes:
\[
\xymatrix@R=1em@C=2em{
	& \XX(\calG;B,\gamma) \ar[dd]\\
	&	\\
\g \ar@{-->}[ruu] \urtwocell\omit{} \ar[r]& \mathfrak{X}(M,\chi)
}
\]
\end{corollary}
\begin{proof}
This follows immediately from the diagram \eqref{eq:proj}.
\end{proof}

If we assume the bundle gerbe over $M$ is $G$-equivariant, with $G$-equivariant connective structure $(B_G,\gamma)$, then the natural $\g$-action on $\calG$,
\(
\xi \mapsto (\xi_X, \xi_P, 0)
\)
is  connection-preserving (in a stricter sense than Definition \ref{d:connexpresmultivf}), and thus gives rise to a morphism of Lie 2-algebras $\mathsf{H}:\g \to \mathbb{X}(\calG; B,\gamma)$ that preserves brackets on the nose (i.e., with $H=0$). Let $\chi_G = \chi-\varpi$ denote the $G$-equivariant 3-curvature of the bundle gerbe. As shown in \cite{callies2016homotopy},  $\chi_G$ gives rise to a homotopy moment map $\mathsf{F}:\g \longrightarrow \mathbb{L}(M,\chi)$, where
\[
\mathsf{F}_0(\xi) = (\xi_M, \varpi(\xi)), \quad \text{and} \quad F(\xi\otimes \zeta) = -\iota_{\xi_M}\varpi(\zeta).
\]
As above, this gives a lift of the $\g$-action as in Corollary \ref{c:2KostantLift}, and 
Theorem \ref{t:g-actsagree} below shows that this lift agrees (up to 2-isomorphism) with the given  $\g$-action on $\calG$.

\begin{theorem} \label{t:g-actsagree}
Let $(M,\chi)$ be a pre-2-plectic manifold, equipped with a $G$-action preserving $\chi$. Suppose that $\chi$ admits a $G$-equivariant extension $\chi_G$ giving rise to a homotopy moment map  $\mathsf{F}:\g \to \mathbb{L}(M,\chi)$.
Let $\calG = (X,P,\mu)$ be a strongly $G$-equivariant bundle gerbe over $M$ with $G$-equivariant connection $(B_G, \gamma)$ and 3-curvature $\chi_G \in \Omega^3_G(M)$. Then the following diagram 2-commutes,
\[
\xymatrix@R=1em@C=2em{
	& {\mathbb{L}(M,\chi)} \ar@{-->}[ddr]^-{\mathsf{E}}	& \\
\rrtwocell\omit{}&	&		\\
\g   \ar[ruu]^-{\mathsf{F}} \ar[rr]_-{\xi \mapsto (\xi_X, \xi_P, 0)} & &  \mathbb{X}(\calG; B,\gamma)  \\	
}
\]
where $\mathsf{E}$ is the butterfly from Theorem \ref{t:invertbutterflymulti-rogers}.
\end{theorem}
\begin{proof}
For $\xi \in \g$, write $B_G(\xi) = B - \nu(\xi)$, where $\nu:\g \to C^\infty(M)$, and write $\chi_G(\xi) = \chi -\varpi(\xi)$, where $\varpi:\g \to \Omega^1(M)$. Recall that since $(B_G,\gamma)$ is a $G$-equivariant connective structure whose 3-curvature is $\chi_G$, we have $\iota_{\xi_P}\gamma = \delta \nu(\xi)$, and $\iota_{\xi_X} B + d\nu(\xi) = \pi^*\varpi(\xi)$. Also, $\iota_{\xi_M} \chi = -d\varpi(\xi)$, so $\varpi(\xi)$ is a Hamiltonian 1-form for $\xi_M$.

The composition $\mathsf{E} \circ \mathsf{F}$ is given by the butterfly 
\[
\xymatrix@R=1em{
0 \ar[dd] \ar[dr]	& 	& \Gamma(A_P) \ar[dd]^d \ar[dl]_-{(0,\lambda)}\\
	& \g\oplus_{L_0} E\ar[dl]^-{\pr_1} \ar[dr]_-{\rho\circ \pr_2}	&	\\
\g	&	& \mathbb{X}(\calG; B, \gamma)_0
}
\]
where $\lambda$ and $\rho$ are the structure maps of $E$, described in the paragraph preceding Theorem \ref{t:invertbutterflymulti-rogers}. Since the homotopy moment map is given by $\mathsf{F}_0(\xi)=(\xi_M,\varpi(\xi))$ in degree 0, elements of $\g\oplus_{L_0} E$ are of the form $(\xi,\mathbf{x},\mathbf{p},\alpha,g)$ such that $\mathbf{x}\sim_\pi \xi_X$, $\delta g = \iota_\mathbf{p} \gamma$, and $\alpha - \iota_\mathbf{x} B - dg =- \pi^*\varpi(\xi)$. The bracket on $\g\oplus_{L_0} E$ is given by 
\[
[(\xi,\mathbf{x},\mathbf{p},\alpha,g),(\zeta,\mathbf{z},\mathbf{r},\alpha,h)] = ([\xi,\zeta], [({\mathbf{x}},{\mathbf{p}},\alpha,g),({\mathbf{z}},{\mathbf{r}},\beta,h)]+(0,0,0,\pi^*\iota_{\xi_M} \varpi(\zeta))).
\]

Consider the linear section $s:\g \to \g\oplus_{L_0} E$ given by $s(\xi) = (\xi, \xi_X, \xi_P, 0, \nu(\xi))$. By \cite[Proposition 3.4]{noohi2013integrating}, the butterfly $\mathsf{E} \circ \mathsf{F}$ is then naturally 2-isomorphic to a morphism of Lie 2-algebras $\mathsf{H}:\g \to \mathbb{X}(\calG; B, \gamma)$, which in degree 0 is given by the $\g$-action on $\calG$. Moreover, the chain homotopy $H:\g \otimes \g \to \Gamma(A_P)$ is given by $H(\xi\otimes \zeta) = [s(\xi),s(\zeta)]-s([\xi,\zeta])$ (where as noted in \cite{noohi2013integrating} we are implicitly using the exactness of the NE-SW sequence in the butterfly). Since $s$ preserves brackets in all but the last coordinate, $H(\xi\otimes \zeta)$ defines a section in the isotropy Lie algebra of $P\toto X$, which can be viewed as a function on $X$, whence
\begin{align*}
 H(\xi\otimes \zeta)	&=\iota_{\zeta_X} \iota_{\xi_X} B + \pi^*\iota_{\xi_M}\varpi(\zeta) - \nu([\xi,\zeta]) \\
	&= \iota_{\zeta_X} \iota_{\xi_X} B + \iota_{\xi_X} (\iota_{\zeta_X} B + d\nu(\zeta) )- \nu([\xi,\zeta]) \\
	&=L_{\xi_X} \nu(\zeta) - \nu([\xi,\zeta]) \\
	&=0
\end{align*}
by $G$-equivariance of $\nu$.
\end{proof}

\subsubsection{Kostant's lifting formula in the quasi-Hamiltonian setting} \label{ss:KostantqHam}

Before discussing the  quasi-Hamiltonian analogue of Kostant's formula \eqref{eq:kostant}, we briefly review the elementary definitions of quasi-Hamiltonian group actions with Lie group-valued moment map.
For the remainder of this section,  $G$ shall denote a connected, compact  Lie group with Lie algebra $\g$, equipped with an invariant inner product $\langle-,-\rangle$. 
 Let $\lmc$ and $\rmc$ denote the left and right invariant Maurer-Cartan forms on  $G$, and let $\eta = \frac{1}{12} \langle \lmc, [\lmc, \lmc] \rangle \in \Omega^3(G)$ denote the Cartan 3-form. We view $G$ as a $G$-manifold equipped with the conjugation action. 

\begin{definition} \cite{alekseev1998lie} \label{d:qHamiltonian}
A \emph{quasi-Hamiltonian $G$-space} is a triple $(M,\omega,\Phi)$ consisting of a $G$-manifold $M$, a $G$-invariant 2-form $\omega \in \Omega^2(M)^G$ and a $G$-equivariant map $\Phi:M \to G$, satisfying
\begin{enumerate}
\item[(i)] $d\omega + \Phi^* \eta =0$, 
\item[(ii)] $\iota_{\xi_M}\omega = \frac{1}{2} \Phi^*\langle \lmc + \rmc, \xi \rangle \quad \text{for all } \xi \in \g$, and
\item[(iii)] at every point $x\in M$, $\ker \omega_x \cap \ker d\Phi_x = \{ 0 \}$.
\end{enumerate}
\end{definition}

Analogous to coadjoint orbits in the dual of the Lie algebra  $\g^*$ as fundamental examples of symplectic manifolds with Hamiltonian $G$-actions, conjugacy classes $\mathcal{C}\hookrightarrow G$ are important examples of quasi-Hamiltonian $G$-spaces. Another important example is the double $G\times G$, with $G$ acting by conjugation on each factor and $G$-valued moment map given by the commutator $(g,h)\mapsto ghg^{-1}h^{-1}$.  Using the \emph{fusion product} operation, these examples form the building blocks for a (finite dimensional) quasi-Hamiltonian description of the moduli space of flat $G$-bundles over a compact surface $\Sigma$, with prescribed holonomies along the boundary curves.  (See \cite{alekseev1998lie} for details.)

In \cite{meinrenken2012twisted}, Meinrenken defines prequantization of quasi-Hamiltonian $G$-spaces using Dixmier-Douady bundles (or simply DD-bundles), which are equivalent to $S^1$-bundle gerbes. More precisely, the bicategory of DD-bundles over a fixed manifold is equivalent to the bicategory of bundle gerbes (and similarly for their $G$-equivariant counterparts)  \cite{krepski2018differential}. Next we recall the notion of prequantization in this context, providing a formulation using bundle gerbes.

To begin, observe that for a quasi-Hamiltonian $G$-space $(M,\omega,\Phi)$, the first condition in Definition \ref{d:qHamiltonian} (together with the fact that $d\eta=0$) is equivalent to the pair $(\omega,\eta)$ defining a closed \emph{relative} differential form in $\Omega^3(\Phi)$. In fact, the first two conditions in Definition \ref{d:qHamiltonian}, are equivalent to the pair $(\omega,\eta_G)$ defining a closed relative $G$-equivariant differential form in $\Omega^3_G(\Phi)$, where \(
\eta_G(\xi) = \eta +  \frac{1}{2} \langle \lmc + \rmc, \xi \rangle
\). 

Recall that a ($G$-equivariant) \emph{prequantization} of a quasi-Hamiltonian $G$-space  $(M,\omega,\Phi)$ is a relative $G$-equivariant bundle gerbe $(Q,\calG)$  for $\Phi$ with relative connective structure $(\gamma_Q; B_G, \gamma)$ whose relative $G$-equivariant 3-curvature is $(\omega,\eta_G)$. That is, a prequantization consists of a $G$-equivariant bundle gerbe\footnote{In this context, it suffices to  consider \emph{strongly} equivariant bundle gerbes (i.e., bundle gerbes admitting compatible $G$-actions on the gerbe data). As noted in Example \ref{e:basicgerbe}, explicit constructions of strongly equivariant bundle gerbes over $G$ exist in the literature.} $\calG$ over $G$ with $G$-equivariant connective structure $(B_G,\gamma) \in \Omega^2_G(X) \oplus \Omega^1(P)^G$ whose $G$-equivariant 3-curvature is $\eta_G$, together with a $G$-equivariant trivialization $(Q,\gamma_Q): \calI_{-\omega} \dashto \Phi^* (\calG,B_G,\gamma)$. 
(The existence of such a trivialization implies the relation $d_G\omega + \Phi^*\eta_G =0$.)

\begin{remark}
Note that a prequantization for $(M,\omega,\Phi)$ exists if and only if the relative cohomology class $[(\omega,\eta_G)] \in H^3_{G}(\Phi;\RR)  $ is integral (i.e., in the image of the coefficient homomorphism $H^3_{G}(\Phi;\RR)  \to H^3_G(\Phi;\ZZ)$.
For related work on the integrality conditions required for prequantization of the moduli space of flat $G$-bundles, see \cite{krepski2008pre,krepski2014prequantization,krepski2013verlinde,meinrenken2003basic}.
\eoe \end{remark}

%

Let $(M,\omega,\Phi)$ be a quasi-Hamiltonian $G$-space. Let $(Q,\gamma_Q): \calI_{-\omega} \dashto \Phi^*(\calG; B, \gamma)$ be a (not necessarily $G$-equivariant) trivialization. Such a trivialization may be viewed as a `non-equivariant'  prequantization of $(M,\omega,\Phi)$. The trivialization $(Q,\gamma_Q)$ gives rise to an invertible butterfly $\mathbb{X}(Q,\gamma_Q):\mathbb{X}(\calI_{-\omega}) \dashto \mathbb{X}(\Phi^*(\calG; B,\gamma))$ (see Section \ref{ss:quasi-isos}).  Moreover, the following Proposition follows immediately.

\begin{proposition} \label{p:square}
Let $(M,\omega,\Phi)$ be a quasi-Hamiltonian $G$-space. Suppose $(\calG, B,\gamma)$ is a bundle gerbe over $G$ with connective structure whose 3-curvature is $\eta$, and let  $(Q,\gamma_Q): \calI_{-\omega} \dashto \Phi^*(\calG, B, \gamma)$ be a trivialization. Then the diagram
\[
\xymatrix{
\LL(M,-d\omega) \ar@{-->}[r] \ar@{=}[d] \drtwocell \omit{}& \mathbb{X}(\calI_{-\omega}) 
\ar@{-->}[d]^{\mathbb{X}(Q,\gamma_Q)} 
\\
\LL(M,\Phi^*\eta) \ar@{-->}[r] & \mathbb{X}(\Phi^*(\calG;B,\gamma))
}
\] 
2-commutes, where the horizontal morphisms are the invertible butterflies from Theorem \ref{t:invertbutterflymulti-rogers}.
\end{proposition}

As in Section \ref{ss:Kostant2plec}, recall that  $G$-equivariant extensions of closed 3-forms give rise to homotopy moment maps \cite{callies2016homotopy}.
For a quasi-Hamiltonian $G$-space $(M,\omega,\Phi)$, we may view $(M,-d\omega)$ as a pre-2-plectic manifold. From condition (i) of Definition \ref{d:qHamiltonian}, it follows that $\LL(M,-d\omega)=\LL(M,\Phi^*\eta)$, and condition (ii) shows that the two $G$-equivariant extensions of $-d\omega$,
\[
-(d_G\omega)(\xi) = -d\omega + \iota_{\xi_M}\omega, \quad \text{and} \quad (\Phi^*\eta_G)(\xi) = \Phi^*\eta +  \frac{1}{2} \Phi^*\langle \lmc + \rmc, \xi \rangle
\]
agree. Therefore, the resulting homotopy moment maps agree. Combining this observation with Proposition \ref{p:square}, we obtain the diagram in the following Theorem, which we consider a quasi-Hamiltonian analogue of the Kostant formula \eqref{eq:kostant}.

\begin{theorem} \label{t:qHamKost}
Let $(M,\omega,\Phi)$ be a quasi-Hamiltonian $G$-space. Suppose $(\calG, B,\gamma)$ is a bundle gerbe over $G$ with connective structure whose 3-curvature is $\eta$, and let  $(Q,\gamma_Q): \calI_{-\omega} \dashto \Phi^*(\calG, B, \gamma)$ be a trivialization. 
The following diagram 2-commutes:
\[
\xymatrix@R=1em{
&\LL(M,-d\omega) \ar@{-->}[r] \ar@{=}[dd] \ddrtwocell \omit{} & \mathbb{X}(\calI_{-\omega}) \ar@{-->}[dd]^{\mathbb{X}(Q,\gamma_Q)} \\
\g \ar[ur] \ar[dr]&&\\
&\LL(M,\Phi^*\eta) \ar@{-->}[r] & \mathbb{X}(\Phi^*(\calG;B,\gamma)).
}
\] 
\end{theorem}
In other words, using the homotopy moment maps arising from the $G$-equivariant extensions of $-d\omega$, we may lift the $\g$-actions on $M$ and $G$ to infinitesimal symmetries on $\calI_{-\omega}$ and $(\calG,B,\gamma)$ that are compatible with the given trivialization $(Q,\gamma_Q)$ coming from a prequantization of $(M,\omega,\Phi)$. (\emph{Cf}.\  the second concluding remark in \cite[Section 13]{callies2016homotopy}.)

As noted in Example \ref{e:basicgerbe}, there are explicit constructions in the literature of (strongly) $G$-equivariant gerbes over compact simple Lie groups $G$ with $G$-equivariant 3-curvature $\eta_G$. This allows for a more explicit description of the butterfly morphisms 
\begin{equation} \label{eq:butterflyactions}
\g \to \mathbb{L}(M,-d\omega) \dashto \mathbb{X}(\calI_{-\omega}), \quad \text{and} \quad 
\g \to \mathbb{L}(M,\Phi^*\eta) \dashto \mathbb{X}(\Phi^*(\calG;B,\gamma))
\end{equation}
appearing in Theorem \ref{t:qHamKost}.

Indeed, assume  $G$ is a compact simple Lie group, and suppose $\calG=(X,P,\mu)$ is a $G$-equivariant bundle gerbe over $G$ with $G$-equivariant connective structure $(B_G, \gamma)$ whose $G$-equivariant 3-curvature is $\eta_G$. Then the pullback $\Phi^*(\calG; B_G, \gamma)$ along a $G$-equivariant map $\Phi:M\to G$  is a  $G$-equivariant bundle gerbe. Denote $\Phi^*\calG$ by $(Y, R,\nu)$.

By Theorem \ref{t:g-actsagree}, it follows that the compositions \eqref{eq:butterflyactions} are 2-isomorphic to infinitesimal actions
by connection-preserving multiplicative vector fields,
$$
\g \to \mathbb{X}(\calI_{-\omega}), \quad \xi \mapsto (\xi_M,0),
$$
(see Example \ref{eg:multivf_trivgerbe} to recall $\mathbb{X}(\calI_{-\omega})$), and 
$$
\g \to \mathbb{X}(\Phi^*(\calG; B, \gamma)), \quad \xi \mapsto (\xi_Y, \xi_R,0),
$$
respectively 
(where we adopt the notation introduced in the discussion preceding Definition \ref{d:connexpresmultivf}).

\appendix

\section{A formula used in the proof of Theorem \ref{t:invertbutterflymulti-rogers}}

We include the verification of a formula used in the proof of Theorem \ref{t:invertbutterflymulti-rogers}.

\begin{proposition} \label{p:jacobiator}
Let $N$ be a manifold, and let $B \in \Omega^2(N)$. Let $X,Y$, and $Z$ denote vector fields on $N$, and suppose that there exist differential 1-forms $a,b$, and $c$  such that $L_X B =da$, $L_Y B=db$, and $L_Z B=dc$. Then,
\[
-\iota_{Z}\iota_{Y}\iota_{X}dB = \iota_{X} (L_{Y} c - L_{Z} b) - \iota_{[Y,Z]} a + \iota_{[Y,Z]} \iota_{X} B + \text{cyc.\ perm.,\ }
\]
\end{proposition}
\begin{proof}
This is a straightforward application of the given relations and the Cartan calculus. Indeed,
\begin{align*}
\iota_X \iota_Y \iota_Z dB &= \iota_X\iota_YL_Z B - \iota_X\iota_Y d\iota_Z B \\
&= \iota_X \iota_Y dc - \iota_X L_Y \iota_Z B + \iota_X d\iota_Y \iota_Z B \\
&= \iota_X \iota_Y dc - \iota_X \iota_{[Y,Z]} B - \iota_X \iota_Z L_Y B + \iota_X d\iota_Y \iota_Z B \\
&= \iota_X \iota_Y dc - \iota_X \iota_{[Y,Z]} B - L_X \iota_Y \iota_Z B \\
&= \iota_X \iota_Y dc - \iota_X \iota_Z db - \iota_X \iota_{[Y,Z]} B + \iota_{[X.Y]} \iota_Z B + \iota_Y L_X \iota_Z B \\
&= \iota_X \iota_Y dc - \iota_X \iota_Z db - \iota_X \iota_{[Y,Z]} B + \iota_{[X.Y]} \iota_Z B + \iota_Y \iota_{[X,Z]} B + \iota_Y \iota Z da \\
&= \iota_X L_Y c - L_X \iota_Y c - \iota_X L_Z b +L_X \iota_Z b + \iota_YL_Z a - L_Y\iota_Z a - \iota_X \iota_{[Y,Z]} B  \\
& \quad -\iota_Z \iota_{[X,Y]} B - \iota_Y \iota_{[Z,X]} B \\
&= \iota_X L_Y c - \iota_{[X,Y]} c - \iota_Y L_X c - \iota_X L_Z b + \iota_{[X,Z]} b + \iota_Z L_X b + \iota_Y L_Z a - \iota_{[Y,Z]} a \\
&\quad - \iota_Z L_Y a - \iota_X \iota_{[Y,Z]}B - \iota_Z \iota_{[X,Y]} B - \iota_Y\iota_{[Z,X]} B 
\end{align*}
\end{proof}


\end{document}